\documentclass[12pt]{amsart}

\usepackage{amsmath,amssymb,latexsym}
\usepackage{dsfont}
\usepackage[T1]{fontenc}
\usepackage{enumerate}
\usepackage{eqnarray}
\usepackage{mathtools}
\usepackage{url}
\usepackage{xcolor}
\usepackage[all]{xy}
\usepackage{pb-diagram,pb-xy}
\usepackage[a4paper,top=3cm, bottom=3cm, inner=3.5cm,outer=2.7cm,headsep=1.5cm]{geometry}
\numberwithin{equation}{section}

\theoremstyle{plain}
\newtheorem{theorem}{Theorem}[section]
\newtheorem{prop}[theorem]{Proposition}
\newtheorem{fact}[theorem]{Fact}
\newtheorem{lemma}[theorem]{Lemma}
\newtheorem{cor}[theorem]{Corollary}
\newtheorem{claim}{Claim}

\theoremstyle{definition}
\newtheorem{defn}[theorem]{Definition}
\newtheorem{remark}[theorem]{Remark}

\newtheorem{problem}[theorem]{Problem}
\newtheorem{expl}[theorem]{Example}
\newtheorem{conj}[theorem]{Conjecture}

\newcommand{\aff}{\operatorname{aff}}

\newcommand{\VC}{\operatorname{VC}}
\newcommand{\vc}{\operatorname{vc}}

\newcommand{\Conv}{\operatorname{Conv}}
\newcommand{\conv}{\operatorname{conv}}

\newcommand{\id}{\operatorname{id}}

\newcommand{\ch}{\operatorname{char}}

\title{Combinatorial properties of non-archimedean convex sets}

\author{Artem Chernikov}
\address{Department of Mathematics, University of California, Los Angeles, Los Angeles, CA 90095, USA}
\email{chernikov@math.ucla.edu}

\author{Alex Mennen}
\address{Department of Mathematics, University of California, Los Angeles, Los Angeles, CA 90095, USA}
\email{alexmennen@math.ucla.edu}

\begin{document}
\begin{abstract}
	We study combinatorial properties of convex sets over arbitrary valued fields. We demonstrate analogs of some classical results for convex sets over the reals (e.g.~the fractional Helly theorem and B\'ar\'any's theorem on points in many simplices), along with some additional properties not satisfied by  convex sets over the reals, including finite breadth and VC-dimension. These results are deduced from a simple combinatorial description of modules over the valuation ring in a spherically complete valued field.
\end{abstract}

\maketitle
\section{Introduction}

Convexity in the context of non-archimedean valued fields was introduced in a series of papers by Monna in 1940's \cite{Monna}, and has been extensively studied since then in non-archimedean functional analysis (see e.g.~the monographs \cite{perez2010locally, schneider2013nonarchimedean} on the subject). Convexity here is defined analogously to the real case, with the role of the unit interval played instead by a valuational unit ball (see Definition \ref{def: convex set}).
Convex subsets of $\mathbb{R}^d$ admit  rich combinatorial structure, including many classical results around the theorems of Helly, Radon, Carath\'eodory, Tverberg, etc. --- we refer to e.g.~\cite{de2019discrete} for a recent survey of the subject. 
In the case of $\mathbb{R}$, or more generally a real closed field, there is  a remarkable parallel between the combinatorial properties of convex and semi-algebraic sets (which correspond to definable sets from the point of view of model theory). They share many (but not all) properties in the form of various restrictions on the possible intersection patterns, including the fractional Helly theorem and existence of (weak) $\varepsilon$-nets.
A well-studied phenomenon in model theory establishes strong parallels between definable sets in $\mathbb{R}$ and in many non-archimedean valued fields such as the $p$-adics $\mathbb{Q}_p$ or various fields of power series (see e.g. \cite{van2014lectures}). 
In this paper we focus on the  combinatorial study of convex sets over general valued fields, trying to understand if there is similarly a parallel theory.
On the one hand, we demonstrate valued field analogs of some classical results for convex sets over the reals (e.g.~the fractional Helly theorem and B\'ar\'any's theorem on points in many simplices). On the other, we establish some additional properties not satisfied by  convex sets over the reals, including finite breadth and VC-dimension. This suggests that in a  sense convex sets over valued fields are the best of both worlds combinatorially, and satisfy various properties enjoyed either by convex or by semialgebraic sets over the reals.

  We give a quick outline of the paper. Section \ref{sec: basic props of conv sets}
 covers some basics concerning convexity for subsets of $K^d$ over an arbitrary valued field $K$, in particular discussing the connection to modules over the valuation ring. These results are mostly standard (or small variations of standard results), 
and can be found e.g.~in \cite{perez2010locally, schneider2013nonarchimedean}  under
the unnecessary assumption that $K$ is spherically complete and $\left(\Gamma,+\right)\subseteq\left(\mathbb{R}_{>0}, \times \right)$; we provide some proofs for completeness.
In Section \ref{sec: O-modules classif} we give a simple combinatorial description of the submodules of $K^d$ over the valuation ring $\mathcal{O}_K$ in the case of a spherically complete field $K$ (Theorem \ref{thm: main Omod pres} and Corollary \ref{cor: descr of conv sets in sph compl}), and an analog for finitely generated modules over arbitrary valued fields (Corollary \ref{cor: class of fg modules}). We also give an example of a convex set over the field of Puiseux series demonstrating that the assumption of spherical completeness is necessary for our presentation in the non-finitely generated case (Example \ref{ex: sph comple necc}). In Section \ref{sec: combinatorics of conv sets} we use this description of modules to deduce various combinatorial properties of the family of convex subsets $\Conv_{K^d}$ of $K^d$ over an arbitrary valued field $K$. First we show that $\Conv_{K^d}$ has breadth $d$ (Theorem \ref{thm: breadth d}), VC-dimension $d+1$ (Theorem \ref{thm: VC dim}), dual VC-dimension $d$ (Theorem \ref{thm: dual VC-dim})  --- in stark contrast, all of these are infinite for the family of convex subsets of $\mathbb{R}^d$ for $d \geq 2$.
On the other hand, we obtain valued field analogs of the following classical results: the family $\Conv_{K^d}$ has Helly number $d+1$ (Theorem \ref{thm: Helly number}), fractional Helly number $d+1$ (Theorem \ref{thm: fractional Helly number}), satisfies a strong form of Tverberg's theorem (Theorem \ref{thm: Tverberg}) and Boros-F\"uredi/B\'ar\'any theorem on the existence of a common point in a positive fraction of all geometric simplices generated by an arbitrary finite set of points in $K^d$ (Theorem \ref{thm: first selection lemma}). Some of the proofs here are adaptations of the classical arguments, and some rely crucially on the finite breadth property specific to the valued field context. Finally, in Section \ref{sec: final quest and dirs} we point out some further applications, e.g.~a valued field analogue of the celebrated $(p,q)$-theorem of Alon and Kleitman \cite{alon1992piercing} (Corollary \ref{cor: pq theorem}), and that all convex sets over a spherically complete field are externally definable in the sense of model theory (Remark \ref{rem: conv ext def}); as well as pose some questions and conjectures. We also discuss some other notions of convexity over non-archimedean fields appearing in the literature in Section \ref{sec: other convexities}, and place our work in the context of the study of abstract convexity spaces in discrete geometry and combinatorics  in Section \ref{sec: convexity spaces}.

\subsection*{Acknowledgements}
We thank the referees for many very helpful literature pointers  and  suggestions on improving the paper. In particular, Sections \ref{sec: other convexities} and \ref{sec: convexity spaces} were added following their suggestions. 
We thank Lou van den Dries for pointing out Monna's work to us, Dave Marker for pointing out Example \ref{ex: naming convex over reals}, and Matthias Aschenbrenner for a helpful conversation. Both authors were partially supported by the NSF CAREER grant DMS-1651321, and Chernikov was additionally supported by a Simons fellowship. 

\section{Preliminaries on convexity over valued fields}\label{sec: basic props of conv sets}
\noindent \textbf{Notation.} For $n \in \mathbb{N}_{\geq 1}$, we write $[n] = \{1, \ldots, n \}$ and $\langle \rangle$ denotes the span in vector spaces. Throughout the paper, $K$ will denote a valued field, with value group
$\Gamma = \Gamma_K$, and valuation $\nu = \nu_K:K\rightarrow \Gamma_{\infty} :=\Gamma\sqcup\left\{ \infty\right\} $,
valuation ring $\mathcal{O}= \mathcal{O}_K = \nu^{-1}\left(\left[0,\infty\right]\right)$,
maximal ideal $\mathfrak{m} = \mathfrak{m}_K =\nu^{-1}\left(\left(0,\infty\right]\right)$,
and residue field\footnote{Also commonly referred to as the ``residue class field'' in the literature.} $k=\mathcal{O}/\mathfrak{m}$. The residue map $\mathcal{O}\rightarrow k$
will be denoted $\alpha\mapsto\bar{\alpha}$. For a ring $R$, $R^{\times}$ denotes its group of units.

The following definition of convexity is analogous to the usual one over $\mathbb{R}$, with the unit interval replaced by the (valuational) unit ball.

\begin{defn}\label{def: convex set}
\begin{enumerate}
	\item For $d \in \mathbb{N}_{\geq 1}$, a set $X \subseteq K^{d}$ is  \emph{convex} if, for any
$n \in \mathbb{N}_{\geq 1}$, $x_{1}, \ldots ,x_{n}\in X$, and $\alpha_{1}, \ldots, \alpha_{n}\in \mathcal{O}$
such that $\alpha_{1}+ \ldots +\alpha_{n}=1$ we have $\alpha_{1}x_{1}+ \ldots +\alpha_{n}x_{n}\in X$ (in the vector space $K^d$).

\item The family of convex subsets of $K^{d}$ will be denoted $\Conv_{K^{d}}$.
\end{enumerate}
\end{defn}
\noindent It is immediate from the definition that the intersection of any collection of convex subsets of $K^d$ is convex. 
\begin{defn}
	Given an arbitrary set $X \subseteq K^d$, its \emph{convex hull} $\conv(X)$ is the convex set given by the intersection of all convex sets containing $X$, equivalently 
$$\conv(X) = \left\{ \sum_{i=1}^{n} \alpha_i x_i :  n \in \mathbb{N}, \alpha_i \in \mathcal{O}, x_i \in X, \sum_{i=1}^{n} \alpha_i = 1\right\}.$$
\end{defn}

\begin{defn}
	A (valuational) \emph{quasi-ball} is a set 
	$B =\left\{ x \in K : \nu(x-c) \in \Delta \right\}$
	 for some $c \in K$ and an upwards closed subset $\Delta$ of $\Gamma_{\infty}$. In this case we say that $B$ is \emph{around} $c$, and refer to $\Delta$ as the quasi-radius of $B$. We say that $B$ is a \emph{closed} (respectively, \emph{open}) \emph{ball} if additionally $\Delta = \{ \gamma \in \Gamma : \gamma \geq r \}$ (respectively, $\Delta = \{ \gamma \in \Gamma : \gamma > r \}$) for some $r \in \Gamma$, and just \emph{ball} if $B$ is either an open or a closed ball (in which case we refer to $r$ as its \emph{radius}).
	 \end{defn}
\begin{remark}\label{rem: any point is center in a ball}
\begin{enumerate}
	\item If the value group $\Gamma$ is Dedekind complete, then every quasi-ball is a ball (except for $K$ itself, which is a quasi-ball of quasi-radius $\Gamma_\infty$).
	\item  Note also that if $B$ is a quasi-ball of quasi-radius $\Delta$ around $c$ and $c' \in B$ is arbitrary, then $B$ is also a quasi-ball of quasi-radius $\Delta$ around $c'$.
	\item In particular, any two quasi-balls are either disjoint, or one of them contains the other.
\end{enumerate}
\end{remark}

\begin{expl}\label{ex: some conv sets}
	\begin{enumerate}
	\item The convex subsets of $K = K^1$ are exactly $\emptyset$ and the quasi-balls (see Proposition \ref{prop: conv iff subm} and Example \ref{ex: submodules of K are quasiballs}).
	\item If $e_1, \ldots, e_d$ is the standard basis of the vector space $K^d$, then 
		$$\conv \left(\left\{0, e_1, \ldots, e_d \right\} \right) = \mathcal{O}^d.$$
		\item The image and the preimage of a convex set under an affine map are  convex. In particular, a translate of a convex set is convex, and a projection of a convex set is convex. (Recall that given two vector spaces $V, W$ over the same field $K$, a map $f: V \to W$ is \emph{affine} if $f(\alpha x + \beta y) = \alpha f(x) + \beta f(y)$ for all $x,y \in V, \alpha, \beta \in K, \alpha + \beta =1$.) 	\end{enumerate}
\end{expl}

One might expect, by analogy with real convexity, that the definition
of a convex set could be simplified to: if $x,y\in X$, $\alpha,\beta\in \mathcal{O}$
such that $\alpha+\beta=1$, then $\alpha x+\beta y\in X$. The following two propositions show that this is
the case if and only if the residue field is not isomorphic to $\mathbb{F}_{2}$, and that in general we have to require closure under $3$-element convex combinations.

\begin{prop}\label{prop: 3 enough for conv} Let $K$ be a valued field and $X \subseteq K^d$.
	 If $X$ is closed under 3-element convex combinations
(in the sense that if $x,y,z\in X$ and $\alpha,\beta,\gamma\in \mathcal{O}$
such that $\alpha+\beta+\gamma=1$, then $\alpha x+\beta y+\gamma z\in X$),
then $X$ is convex.
\end{prop}
\begin{proof}
	Suppose $X$ is closed under 3-element convex combinations.
We will show by induction on $n$ that then $X$ is closed under $n$-element
convex combinations. Let $n \geq 3$, $x_{1},\ldots,x_{n}\in X$ and $\alpha_{1}, \ldots ,\alpha_{n}\in \mathcal{O}$
such that $\alpha_{1}+ \ldots +\alpha_{n}=1$ be given. Then one of the following
two cases holds.

\begin{enumerate}
	\item [Case 1:] $\alpha_{1}+\alpha_{2}\in \mathcal{O}^{\times}$. 
	
\noindent Then $\frac{\alpha_{1}}{\alpha_{1}+\alpha_{2}}$
and $\frac{\alpha_{2}}{\alpha_{1}+\alpha_{2}}$ are elements of $\mathcal{O}$
that sum to $1$, so $$\frac{\alpha_{1}}{\alpha_{1}+\alpha_{2}}x_{1}+\frac{\alpha_{2}}{\alpha_{1}+\alpha_{2}}x_{2}\in X$$ by assumption. But then 
$$\alpha_{1}x_{1}+\ldots+\alpha_{n}x_{n}=\left(\alpha_{1}+\alpha_{2}\right)\left(\frac{\alpha_{1}}{\alpha_{1}+\alpha_{2}}x_{1}+\frac{\alpha_{2}}{\alpha_{1}+\alpha_{2}}x_{2}\right)+\alpha_{3}x_{3}+ \ldots +\alpha_{n}x_{n}\in X$$
by the induction hypothesis, as it is a convex combination
of $n-1$ elements of $X$.

\item [Case 2:]$\alpha_{1}+\alpha_{2}\in\mathfrak{m}$.

\noindent Then, as $\nu \left( \sum_{i=1}^{n} \alpha_i \right) = 0$, there must exist some $i$ with
$3 \leq i \leq n$ such that $\alpha_{i}\in \mathcal{O}^{\times}$. Hence $\alpha_{1}+\alpha_{2}+\alpha_{i}\in \mathcal{O}^{\times}$,
so $\frac{\alpha_{1}}{\alpha_{1}+\alpha_{2}+\alpha_{i}}$, $\frac{\alpha_{2}}{\alpha_{1}+\alpha_{2}+\alpha_{i}}$,
and $\frac{\alpha_{i}}{\alpha_{1}+\alpha_{2}+\alpha_{i}}$ are elements
of $\mathcal{O}$ that sum to $1$. Thus 
$$\left(\frac{\alpha_{1}}{\alpha_{1}+\alpha_{2}+\alpha_{i}} \right)x_{1}+\left(\frac{\alpha_{2}}{\alpha_{1}+\alpha_{2}+\alpha_{i}} \right)x_{2}+\left(\frac{\alpha_{i}}{\alpha_{1}+\alpha_{2}+\alpha_{i}} \right)x_{i}\in X$$
by assumption, and so
\begin{gather*}
	\alpha_{1}x_{1}+ \ldots +\alpha_{n}x_{n}= \\
	\left(\alpha_{1}+\alpha_{2}+\alpha_{i}\right)\left(\frac{\alpha_{1}}{\alpha_{1}+\alpha_{2}+\alpha_{i}}x_{1}+\frac{\alpha_{2}}{\alpha_{1}+\alpha_{2}+\alpha_{i}}x_{2}+\frac{\alpha_{i}}{\alpha_{1}+\alpha_{2}+\alpha_{i}}x_{i}\right) \\
	+\alpha_{3}x_{3} + \ldots + \alpha_{i-1}x_{i-1}+\alpha_{i+1}x_{i+1}+ \ldots +\alpha_{n}x_{n}\in X
\end{gather*}
by the induction hypothesis, as it is a convex combination
of $n-2$ elements of $X$. \qedhere
\end{enumerate}
\end{proof}

\begin{prop}
For any valued field $K$, the following are equivalent:

\begin{enumerate}
	\item for every $d \geq 1$, every set in $K^d$ that is closed under 2-element convex combinations
is convex;
\item the residue field $k$ is not isomorphic to $\mathbb{F}_{2}$.
\end{enumerate}
\end{prop}
\begin{proof}
\emph{(1) implies (2).}	If $k=\mathbb{F}_{2}$, consider the set
	$$X:=\left\{ \left(a_{1},a_{2},a_{3}\right)\mid a_{1},a_{2},a_{3}\in \mathcal{O},\,\exists i\,a_{i}\in\mathfrak{m}\right\} \subseteq K^{3}.$$
We claim that $X$ is closed under $2$-element convex combinations. That is, given arbitrary  $\left(a_{1},a_{2},a_{3}\right), \left(b_{1},b_{2},b_{3}\right)\in X$ 
and $\alpha,\beta\in \mathcal{O}$ with $\alpha+\beta=1$, we must show
that $\alpha\left(a_{1},a_{2},a_{3}\right)+\beta\left(b_{1},b_{2},b_{3}\right)\in X$. We have
$\bar{\alpha}+\bar{\beta}=1$ in $k = \mathbb{F}_2$, so necessarily one of $\bar{\alpha}$ and $\bar{\beta}$
is $1$ and the other is $0$. Without loss of generality $\bar{\alpha}=1$ and $\bar{\beta}=0$.
Then $\beta\in\mathfrak{m}$. By definition of $X$, $a_{i}\in\mathfrak{m}$ for some $i$.
Then $\alpha a_{i}\in\mathfrak{m}$, and $\beta b_{i}\in\mathfrak{m}$ as $b_i \in \mathcal{O}$,
so $\alpha a_{i}+\beta b_{i}\in\mathfrak{m}$. Thus $\left(\alpha a_{1}+\beta b_{1},\alpha a_{2}+\beta b_{2},\alpha a_{3}+\beta b_{3}\right)\in X$. However $X$ is not convex: for an arbitrary $a \in \mathfrak{m}$ we have $(0,0,0), (1,0,0), (0,1,1) \in  X$, $1,-1 \in \mathcal{O}$, but $(-1)(0,0,0) + 1(1,0,0) + 1(0,1,1) = (1,1,1) \notin X$. (This example can be modified to work in $K^{2}$.)

\noindent\emph{(2) implies (1).} If $k\not\cong\mathbb{F}_{2}$, suppose $X$ is closed under $2$-element
convex combinations. By Proposition \ref{prop: 3 enough for conv}, we only need to check that it is then closed under $3$-element
convex combinations. Let $x,y,z\in X$, and $\alpha,\beta,\gamma\in \mathcal{O}$
such that $\alpha+\beta+\gamma=1$. Then one of the following two
cases holds.
\begin{enumerate}
	\item [Case 1:]
At least one of $\alpha+\beta, \beta+\gamma, \alpha+\gamma$ is an element of $\mathcal{O}^{\times}$.

\noindent Without loss of generality, $\alpha+\beta\in \mathcal{O}^{\times}$. Then
$\frac{\alpha}{\alpha+\beta}x+\frac{\beta}{\alpha+\beta}y\in X$ by assumption,
and thus 
$$\alpha x+\beta y+\gamma z=\left(\alpha+\beta\right)\left(\frac{\alpha}{\alpha+\beta}x+\frac{\beta}{\alpha+\beta}y\right)+\gamma z\in X.$$
\item [Case 2:] $\alpha+\beta,\beta+\gamma,\alpha+\gamma\in\mathfrak{m}$.

\noindent In the residue field, $\bar{\alpha}+\bar{\beta}=\bar{\beta}+\bar{\gamma}=\bar{\alpha}+\bar{\gamma}=0$,
and $\bar{\alpha}+\bar{\beta}+\bar{\gamma}=1$, hence necessarily $\bar{\alpha}=\bar{\beta}=\bar{\gamma}=1$,
and $\ch \left(k\right)=2$. Since $k\not\cong\mathbb{F}_{2}$,
there is $\delta\in \mathcal{O}$ such that $\bar{\delta}\notin\left\{ 0,1\right\} $. Then 
$\bar{\alpha}+\bar{\delta}=1+\bar{\delta}\neq0$ and $\bar{\beta}-\bar{\delta}+\bar{\gamma}=\bar{\delta}\neq0$,
so 
\begin{gather*}
	\alpha x+\beta y+\gamma z=\\
	\left(\alpha+\delta\right)\left(\frac{\alpha}{\alpha+\delta}x+\frac{\delta}{\alpha+\delta}y\right)+\left(\beta-\delta+\gamma\right)\left(\frac{\beta-\delta}{\beta-\delta+\gamma}y+\frac{\gamma}{\beta-\delta+\gamma}z\right)\in X.
\end{gather*}
\end{enumerate}
\end{proof}

The following proposition gives a very strong form of Radon's theorem (not only we obtain a partition into two sets with intersecting convex hulls, but moreover one of the points is in the convex hull of the other ones).
\begin{prop}\label{prop: Radon}
	Let $K$ be a valued field. For any $d+2$ points $x_{1}, \ldots ,x_{d+2}\in K^{d}$,
one of them is in the convex hull of the others.
\end{prop}
\begin{proof}
	There exist $a_{1}, \ldots, a_{d+2}\in K$, \emph{not all $0$}, such that
$\sum_{i=1}^{d+2}a_{i}x_{i}= 0$  and $\sum_{i=1}^{d+2}a_{i}=0$ (because those are $d+1$ linear equations on $d+2$ variables, as we are working in $K^d$).
Let $i\in\left[d+2\right]$ be such that $\nu\left(a_{i}\right)$
is minimal among $\nu(a_1), \ldots, \nu(a_{d+2})$, in particular $a_i \neq 0$. Then $x_{i}=\sum_{j\neq i}\frac{-a_{j}}{a_{i}}x_{j}$,
and this is a convex combination: for $i \neq j$ we have $\frac{-a_{j}}{a_i}  \in \mathcal{O}$ (as $\nu \left(\frac{-a_{j}}{a_i} \right) = \nu(a_j) - \nu(a_i) \geq 0$ by the choice of $i$) and  $\sum_{j \neq i}\frac{-a_j}{a_i} = \frac{-\sum_{j \neq i} a_j}{a_i} = \frac{a_i}{a_i} = 1$. 
\end{proof}

\noindent By a repeated application of Proposition \ref{prop: Radon} we immediately get a very strong form of Carath\'eodory's theorem:
\begin{cor}\label{cor: Carath}
	Let $K$ be a valued field. Then the convex hull of any finite set in $K^d$ is already given by the convex hull of at most $d+1$ points from it.
\end{cor}

Convex sets over valued fields have a natural algebraic characterization. 

\begin{prop}\label{prop: conv iff subm}
\begin{enumerate}
\item A subset $C \subseteq K^{d}$ is an $\mathcal{O}$-submodule
of $K^{d}$ if and only if it is convex and contains $0$.

	\item 	Nonempty convex subsets of $K^{d}$ are precisely the  translates
of $\mathcal{O}$-submodules of $K^d$.
\end{enumerate}
\end{prop}
\begin{proof}
	(1) First, $\mathcal{O}$-submodules of $K^d$ are clearly convex and contain $0$.
Now suppose $C\subseteq K^{d}$ is convex and $0 \in  C$.
Then for any $\alpha\in \mathcal{O}$ and $x\in C$, $\alpha x=\alpha x+\left(1-\alpha\right) 0\in C$.
And for any $x,y\in C$, $x+y=1\cdot x+1\cdot y-1\cdot 0\in C$. Therefore
$C$ is an $\mathcal{O}$-submodule.
(2) Given a non-empty convex $C \subseteq K^d$, we can choose $a \in K^d$ such that the translate $C + a$ contains $0$, and it is still convex, hence $C+a$ is an $\mathcal{O}$-submodule of $K^d$ by (1).
\end{proof}

\begin{expl}\label{ex: submodules of K are quasiballs}
	Let $C$ be an $\mathcal{O}$-submodule of $K$, and take $\Delta := \nu(C)$. Then $\Delta$ is non-empty because it contains $\infty=\nu(0)$, and upward-closed because for $\gamma\in\Delta$ and $\delta>\gamma$, there is $x\in C$ with $\nu(x)=\gamma$, and $\alpha\in K$ with $\nu(\alpha)=\delta-\gamma$; then $\alpha x\in C$ and $\nu(\alpha x)=\delta$. Clearly $C\subseteq\{x\in K\mid\nu(x)\in\Delta\}$ by definition of $\Delta$. To show $C\supseteq\{x\in K\mid\nu(x)\in\Delta\}$, given any $x\in K$ with $\nu(x)\in\Delta$, there is $y \neq 0\in C$ with $\nu(y)=\nu(x)$, and $\frac{x}{y}\in \mathcal{O}$, so $x=\frac{x}{y}y\in C$. Thus $C=\{x\in K\mid\nu(x)\in\Delta\}$ is a quasi-ball around $0$.	
\end{expl}

\begin{cor}
	The convex hull of any finite set in $K^{d}$ is the
image of $\mathcal{O}^{d}$ under an affine map.
\end{cor}
\begin{proof}
	By Corollary \ref{cor: Carath}, the convex hull of a finite subset of $K^d$ is the convex
hull of some $d+1$ points $x_{0}, \ldots, x_{d}$ from it (possibly with  $x_i = x_j$ for some $i,j$). Let $e_1, \ldots, e_d$ be the standard basis for $K^d$, and let $f$ be an  affine map $f:K^{d}\rightarrow K^{d}$
such that $f(0)=x_{0}$ and $f\left(e_{i}\right)=x_{i}$
for $1\leq i\leq d$ (can take $f$ to be the composition of two affine maps: the linear map sending $e_i$ to $x_{i} - x_0$ for $1 \leq i \leq d$, and translation by $x_0$). Then we have $\conv\left(\left\{ x_{0}, \ldots, x_{d}\right\} \right)=f\left(\conv\left\{ 0,e_{1}, \ldots,e_{d}\right\} \right)=f\left(\mathcal{O}^{d}\right)$ (by Example \ref{ex: some conv sets}(2)).
\end{proof}

\begin{prop}
	For any convex $C \subseteq K^{d}$ and $a\in K^{d}$,
the translate $C+a:=\left\{ x+a \mid x\in C\right\} $ is either equal
to or disjoint from $C$.
\end{prop}
\begin{proof}
	If $x\in C\cap\left(C+a\right)$, then $\forall y\in C$ $y+a=y+x-\left(x-a\right)\in C$,
since that is a convex combination, and conversely, if $y+a\in C$
then $y=\left(y+a\right)-x+\left(x-a\right)\in C$. 
\end{proof}

\begin{defn}
	
Given a valued field $K$, by a \emph{valued $K$-vector space} we mean a $K$-vector space $V$ equipped with a surjective map $\nu = \nu_V: V \to \Gamma_{\infty} = \Gamma \cup \{\infty\}$ such that $\nu(x) = \infty$ if and only if $x = 0$, $\nu(x+y) \geq \min\{ \nu(x), \nu(y) \}$ and  $\nu(\alpha x) = \nu_K(\alpha) + \nu(x)$ for all $x,y \in V$ and $\alpha \in K$.
\end{defn}
\begin{remark}
Here we restrict to the case when $V$ has the same value group as $K$, and refer to \cite{fuchs1975vector} for a more general treatment (see also \cite[Section 6.1.3]{johnson2016fun}, \cite[Section 2.5]{hrushovski2014imaginaries} or \cite[Section 2.3]{aschenbrenner2017asymptotic}).\end{remark}
\noindent By a morphism of valued $K$-vector spaces we mean a morphism of vector spaces preserving valuation. If $V$ and $W$ are valued $K$-vector spaces, their direct sum $V \oplus W$ is the direct sum of the underlying vector spaces equipped with the valuation $\nu(x,y) := \min \{ \nu_V(x), \nu_{W}(y)\}$.
In particular, the vector space $K^{d}$ is a valued $K$-vector space with respect to the valuation $\nu_{K^d}:K^{d}\rightarrow\Gamma_{\infty}$
given by 
$$\nu_{K^d}\left(x_{1}, \ldots ,x_{d}\right) := \min\left\{ \nu_K\left(x_{1}\right), \ldots , \nu_K \left(x_{d}\right)\right\}.$$
Note that for any scalar $\alpha \in K$ and vector $v \in K^d$ we have $\nu_{K^d}(\alpha v) = \nu_K(\alpha) + \nu_{K^d}(v)$.
  By a \emph{(valuational) ball} in $K^d$ we mean a set of the form $\{ x \in K^{d} : \nu_{K^d}(x - c) \square r \}$ for some center $c \in K^d$, radius $r \in \Gamma \cup \{\infty\}$ and $\square \in \{ >, \geq \}$ (corresponding to open or closed ball, respectively). The collection of all open balls forms a basis for the \emph{valuation topology} on $K^d$ turning it into a topological vector space. Note that due to the ``ultra-metric'' property of valuations, every open ball is also a closed ball, and vice versa. Equivalently, this topology on $K^d$ is just the product topology induced from the valuation topology on $K$.
   
Recall that the \emph{affine span $\aff(X)$} of a set  $X \subseteq K^d$ is the intersection of all affine sets (i.e.~translates of vector subspaces of $K^d$) containing $X$, equivalently 
$$\aff(X) = \left\{ \sum_{i=1}^{n} \alpha_i x_i :  n \in \mathbb{N}_{\geq 1}, \alpha_i \in K, x_i \in X, \sum_{i=1}^{n} \alpha_i = 1\right\}.$$
We have $\conv(X) \subseteq \aff(X)$ for any $X$.

\begin{prop}
	Any convex set in $K^{d}$ is open in its affine span.
\end{prop}
\begin{proof}
	For $x\in C\subseteq K^{d}$, $C$ convex, let $d'\leq d$
be the dimension of the affine span of $C$, and let $y_{1},\ldots,y_{d'}\in C$ be 
such that $x,y_{1},\ldots,y_{d'}$ are affinely independent, and thus
have the same affine span as $C$. Then the map $\left(\alpha_{1},\ldots,\alpha_{d'}\right)\mapsto x+\alpha_{1}\left(y_{1}-x\right)+\ldots+\alpha_{d'}\left(y_{d'}-x\right)$
is a homeomorphism from $K^{d'}$ to the affine span of $C$, and
sends $\mathcal{O}^{d'}$ (which is open in $K^{d'}$) to a neighborhood
of $x$ contained in $C$. 
\end{proof}

\begin{cor}
Convex sets in $K^{d}$ are closed.
\end{cor}
\begin{proof}
For convex $C\subseteq K^{d}$ and $x\in \aff\left(C\right)\setminus C$,
$C+x$ is an open subset of $\aff\left(C\right)$ that is
disjoint from $C$, so $C$ is a closed subset of its affine span,
and hence closed in $K^{d}$, since affine subspaces are closed. 
\end{proof}

\section{Classification of $\mathcal{O}$-submodules of $K^d$}\label{sec: O-modules classif}

In this section we provide a simple description for the $\mathcal{O}$-submodules of $K^d$ over a spherically complete valued field $K$ (and over an arbitrary valued field $K$ in the finitely generated case). Combined with the description of convex sets in terms of $\mathcal{O}$-submodules from Section \ref{sec: basic props of conv sets}, this will allow us to establish various combinatorial properties of convex sets over valued fields in the next section. In the following lemma, the construction of the valuation $\nu$  is a special case of the standard construction of the quotient norm, when modding out a normed space by a closed subspace, while the second part is more specific to our situation.

\begin{lemma}\label{lem: val pres emb quotient}
Let $K$ be a valued field, and $V\subseteq K^{d}$ a subspace.
Then the quotient vector space $K^d/V$ is a valued  $K$-vector space equipped with the valuation 
$$\nu\left(u\right):=\max\left\{ \nu_{K^d}\left(v\right)\mid\pi\left(v\right)=u, v \in K^d \right\},$$
for $u \in K^d/V$, where $\pi: K^d \rightarrow K^d/V$ is the projection map (and the maximum is taken in $\Gamma_{\infty}$).
If $\dim(V) = n$, then $K^{d}/V\cong K^{d-n}$ as valued $K$-vector spaces, and there is
a valuation preserving embedding of $K$-vector spaces $f: K^{d}/V\hookrightarrow K^{d}$  so that $\pi \circ f = \id_{K^d/V}$.
\end{lemma}
\begin{proof}
	First we prove the lemma for $n=1$. Let $V\subseteq K^{d}$ be one-dimensional.  There exists 
 $i\in\left[d\right]$ such that $\nu_{K^d}\left(\left(x_{1},\ldots,x_{d}\right)\right)=\nu_K\left(x_{i}\right)$
for all $\left(x_{1},\ldots,x_{d}\right)\in V$ (indeed, if $\nu_K(x_i) = \min \{\nu_K(x_1), \ldots, \nu_K(x_d)\}$ for some $(x_1, \ldots, x_d) \in V$, then we also have $\nu_K(\alpha x_i) = \nu_K(\alpha) + \nu_K (x_i) = \nu_K(\alpha) + \min \{\nu_K(x_1), \ldots, \nu_K(x_d)\} = \min \{\nu_K(\alpha x_1), \ldots, \nu_K( \alpha x_d)\}$ for any $\alpha \in K$). Given any $ \left(x_{1}, \ldots,x_{d}\right)\in K^{d}$
with $x_{i}=0$ and $ \left(y_{1}, \ldots,y_{d}\right)\in V$, we have 
\begin{gather}
	\nu_{K^d}\left(x_{1}+y_{1}, \ldots, x_{d}+y_{d}\right)=\min_{j\in\left[d\right]}\left\{ \nu_K\left(x_{j}+y_{j}\right)\right\} =  \label{eq: MainThmeqNew} \\
	\min\left \{ \nu_K\left(y_{i}\right),\min_{j\neq i}\left\{ \text{\ensuremath{\nu_K\left(x_{j}+y_{j}\right)}}\right\} \right \} \leq\nu_K\left(y_{i}\right)=\nu_{K^d}\left(y_{1}, \ldots, y_{d}\right). \nonumber
\end{gather}
Now consider an arbitrary affine translate $x + V$ of $V$, $x = (x_1, \ldots, x_d) \in K^d$. Then there exists $x' = (x'_1, \ldots, x'_d) \in x + V$ so that $x'_i = 0$. Indeed, fix any  $0 \neq y' \in V$, then $V = \left\{ \alpha y' : \alpha \in K \right\}$. Take $\alpha' := - \frac{x_i}{y'_i}$ (note that, by the choice of $i$, $y' \neq 0 \Rightarrow \nu_{K^d}(y') \neq \infty \Rightarrow \nu_{K}(y'_i) \neq \infty \Rightarrow y'_i \neq 0$), and let $x' := x + \alpha ' y'$.
We claim that $\nu_{K^d} (x') = \max \left\{ \nu_{K^d}(z) :  z \in x + V \right\}$,  in particular the valuation $\nu$ on $K^d/V$ is well-defined. Indeed, $x + V = x' + V$, so fix any $y \in V$. If $\nu_{K^d}(x') < \nu_{K^d}(x' + y) $, we must necessarily have $\nu_{K^d}(x') = \nu_{K^d}(y)$, but by \eqref{eq: MainThmeqNew} we have $\nu_{K^d}(x' + y)  \leq \nu_{K^d}(y) $, so $\nu_{K^d}( y) < \nu_{K^d}(y)  $ --- a contradiction; thus $\nu_{K^d}(x') \geq \nu_{K^d}(x' + y) $.

%Thus the maximum of the valuations of elements of any given affine
%translate $x + V = \left\{ x + \alpha y : \alpha \in K \right\}$ of $V$, where $x \in K^d$ and $y$ is a generator of $V$, is achieved by an element of that translate with
%zero $i$th coordinate, in particular the valuation $\nu$ on $K^d/V$ is well-defined.

Let $K' := \left\{ \left(x_{1},\ldots,x_{d}\right)\in K^{d}\mid x_{i}=0\right\} $,  then we have  $K^d = V \oplus K'$ as vector spaces, hence the projection of $K^d$ onto $K'$
along $V$ induces an isomorphism between $K^{d}/V$ and $K'$,
which in turn is naturally isomorphic to $K^{d-1}$, and these isomorphisms preserve the valuation and give the
desired embedding $f: K^{d}/V\rightarrow K^{d}$. 
The general case follows
by induction on $n$ using the vector space isomorphism theorems.
\end{proof}

We recall an appropriate notion of completeness for valued fields. Recall that a family $\{ C_i : i \in I \}$ of subsets of a set $X$ is \emph{nested} if for any $i,j \in I$, either $C_i \subseteq C_j$ or $C_j \subseteq C_i$.
\begin{defn}
	A valued field $K$ is \emph{spherically complete} if every nested family of (closed or open) valuational balls has non-empty intersection.
\end{defn} 

For the following standard fact, see for example \cite[Theorem 5 in Section II.3 + Theorem 8 in section II.6]{schilling1950theory}.
\begin{fact}\label{fac: sph compl exists}
	Every valued field $K$ (with valuation $\nu_K$, value group $\Gamma_K$ and residue field $k_K$) admits a \emph{spherical completion}, i.e.~a valued field $\widetilde{K}$ (with valuation $\nu_{\widetilde{K}}$, value group $\Gamma_{\widetilde{K}}$ and residue field $k_{\widetilde{K}}$) so that:
	\begin{enumerate}
		\item $\widetilde{K}$ is an \emph{immediate extension} of $K$, i.e.~$\widetilde{K}$ is a field extension of $K$, $\nu_{\widetilde{K}}\restriction_{K} = \nu_K$, $\Gamma_{\widetilde{K}} = \Gamma_K$ and $k_{\widetilde{K}} = k_{K}$;
		\item $\widetilde{K}$ is spherically complete.
	\end{enumerate}  
\end{fact}

We remark that in general a valued field might have multiple non-isomorphic spherical completions.

\begin{lemma}\label{lem: inters of conv is nonempty}
	If $K$ is spherically complete, then  every nested family of non-empty convex subsets of $K^{d}$  has a non-empty intersection.
\end{lemma}

\begin{proof}
	By induction on $d$. For $d=1$, 	let $\left\{C_i\right\}_{i\in I}$ be a nested family of nonempty convex sets, so each $C_i$ is a quasi-ball (see Example \ref{ex: some conv sets}(1)). If there exists some $i \in I$ so that $C_i$ is the smallest of these under inclusion then any element of $C_i$ is in the intersection of the whole family. Hence we may assume that for each $i\in I$
	there exists some $i' \in I$ such that $C_{i'}\subsetneq C_i$. Let $\Delta_i$ and $\Delta_{i'}$ be the quasi-radii of $C_i$ and $C_{i'}$, respectively. We may assume that both quasi-balls are around the same point $x_i \in C_{i'}$ (by Remark \ref{rem: any point is center in a ball}), hence necessarily $\Delta_{i'} \subsetneq \Delta_i $. Let $r_i \in \Delta_i \setminus \Delta_{i'}$, and let $C'_i$ be a (open or closed) ball of radius $r_i$ around $x_i$. 	We have $C'_i \subseteq  C_i$, so if $\bigcap_{i\in I}C'_i$ is nonempty, then so is $\bigcap_{i\in I}C_i$. Hence it is sufficient to show that $\left\{C'_i\right\}_{i\in I}$ is nested, and then the intersection is non-empty by spherical completeness of $K$. By construction for any $i,j\in I$ there exists some $\ell\in I$ such that $C_\ell\subseteq C'_i\cap C'_j$, so $C'_i$ and $C'_j$ have non-empty intersection, and are thus nested as they are balls.

	For $d \geq 2$, let $\left\{ C_{i}\right\} _{i\in I}$
be a nested family of nonempty convex sets, and let $\pi_{1}:K^{d}\rightarrow K$
be the projection onto the first coordinate. Then $\left\{ \pi_{1}\left(C_{i}\right)\right\} _{i\in I}$
is a nested family of nonempty convex sets in $K$, hence has an
intersection point $x$. Then $\left\{ \pi_{1}^{-1}\left(x\right)\cap C_{i}\right\} _{i\in I}$
is a nested family of nonempty convex sets in $\pi_{1}^{-1}\left(x\right)\cong K^{d-1}$,
which is nonempty by the induction hypothesis.
\end{proof}

\begin{lemma}\label{lem: interm conv set}
	If $C\subseteq K^{d}$ is an $\mathcal{O}$-module, and $\gamma\in\Gamma_{\infty}$,
then the set
$$X_{\gamma} =\left\{ \left(x_{1}, \ldots,x_{d-1}\right)\in \mathcal{O}^{d-1}\mid\exists\alpha\in K\;\nu\left(\alpha\right)=\gamma,\;\left(\alpha,\alpha x_{1},\ldots,\alpha x_{d-1}\right)\in C\right\} $$
is convex.
\end{lemma}

\begin{proof}
	Let $x =\left(x_{1},\ldots,x_{d-1}\right),y=\left(y_{1}, \ldots, y_{d-1}\right),z=\left(z_{1}, \ldots ,z_{d-1}\right)\in X_{\gamma}$ and $\beta_1, \beta_2, \beta_3 \in \mathcal{O}$ with $\beta_1 + \beta_2 + \beta_3 = 1$ be arbitrary.
Then there exist some $\alpha_{1},\alpha_{2},\alpha_{3}\in K$ with $\nu\left(\alpha_{i}\right)=\gamma$ so that
$$\left(\alpha_{1},\alpha_{1}x_{1},\ldots,\alpha_{1}x_{d-1}\right), \left(\alpha_{2},\alpha_{2}y_{1}, \ldots ,\alpha_{2}y_{d-1}\right),\left(\alpha_{3},\alpha_{3}z_{1}, \ldots,\alpha_{3}z_{d-1}\right)\in C.$$ 
Taking $\alpha:=\alpha_{1}$, we have 
$$x':=\left(\alpha,\alpha x_{1}, \ldots, \alpha x_{d-1}\right), y' := \left(\alpha,\alpha y_{1}, \ldots, \alpha y_{d-1}\right), z' := \left(\alpha,\alpha z_{1}, \ldots, \alpha z_{d-1}\right)\in C,$$
as for every $i \in [3]$, $\frac{\alpha}{\alpha_{i}}\in \mathcal{O}$, and hence $\frac{\alpha}{\alpha_{i}}v\in C$
for any $v\in C$ as $C$ is an $\mathcal{O}$-module. Using this and convexity  of $C$ we thus have
\begin{gather*}
\Big(\alpha, \alpha( \beta_1 x_1 + \beta_2 y_1 + \beta_3 z_1), \ldots,  \alpha( \beta_1 x_{d-1} + \beta_2 y_{d-1} + \beta_3 z_{d-1})\Big) =\\
	\beta_1\left(\alpha,\alpha x_{1}, \ldots,\alpha x_{d-1}\right)+ \beta_2 \left(\alpha,\alpha y_{1}, \ldots ,\alpha y_{d-1}\right)+\beta_3 \left(\alpha,\alpha z_{1},\ldots,\alpha z_{d-1}\right) =\\
	 \beta_1 x' + \beta_2 y' + \beta_3 z' \in C.
\end{gather*}
This shows that $\beta_1 x + \beta_2 y + \beta_3 z \in X_{\gamma}$, and hence that $X_{\gamma}$ is convex by Proposition \ref{prop: 3 enough for conv}.
\end{proof}

Combining the lemmas, we obtain the following description of the $\mathcal{O}_K$-submodules of $K^d$ for spherically complete $K$.

\begin{theorem}\label{thm: main Omod pres}
	Suppose $K$ is a spherically complete valued field, $d \in \mathbb{N}_{\ge 1}$, and let $C\subseteq K^{d}$
be an $\mathcal{O}$-submodule. Then there exists a complete flag of vector subspaces $\left\{ 0\right\} \subsetneq F_{1}\subsetneq \ldots \subsetneq F_{d}=K^{d}$
and a decreasing sequence of nonempty, upwards-closed subsets $ \Delta_{1}\supseteq\Delta_{2}\supseteq \ldots \supseteq\Delta_{d}$ of $\Gamma_{\infty}$
such that 
$$C=\left\{ v_{1}+ \ldots +v_{d}\mid v_{i}\in F_{i},\;\nu\left(v_{i}\right)\in\Delta_{i}\right\}.$$
\end{theorem}
\begin{remark}\label{rem: Delta1}
	If $F_i, \Delta_i$ satisfy the conclusion of Theorem \ref{thm: main Omod pres} for $C$, then $\nu_{K^d}(C \cap F_1) = \nu_{K^d}(C) = \Delta_{1}$.
	
	Indeed, any $v \in C$ is of the form $v=v_1 + \ldots + v_d$ with $v_i \in F_i$, $\nu(v_i) \in \Delta_i$ and $\Delta_1 \supseteq \Delta_i$ for all $i \in [d]$, hence $\nu(v) \geq \min\left\{\nu(v_i) : i \in [d]\right\} \in \Delta_1$, hence $\nu(v) \in \Delta_1$ as $\Delta_1$ is upwards closed, so $\nu(C) \subseteq \Delta_1$. Conversely, assume $\gamma \in \Delta_1$. If $\gamma = \infty$, then $\nu(0) = \infty$ and $0 \in F_1$. So assume $\gamma \in \Gamma$ and let $v$ be any non-zero vector in $F_1$, in particular $\delta := \nu(v) \in \Gamma$. Taking $\alpha \in K$ so that $\nu_K(\alpha) = \gamma - \delta$, we have $\alpha v \in F_1$ and $\nu_{K^d}(\alpha v) = \nu_{K}(\alpha) + \nu_{K^d}(v) = \gamma$. Note also that $\alpha v = v_1 + \ldots + v_d$ with $v_1 := \alpha v, v_i := 0$ for $2 \leq i \leq d$, in particular $v_i \in F_i$ and $\nu(v_i) \in \Delta_i$, so $\alpha v \in  C$, hence $\Delta_1 \subseteq \nu(F_1 \cap C)$.
\end{remark}

\begin{proof}[Proof of Theorem \ref{thm: main Omod pres}]
	By induction on $d$. 
	For $d=1$, every $\mathcal{O}$-submodule of $K$ is a quasi-ball $C = \{ x \in K : \nu(x) \in \Delta \}$ for some upwards-closed $\Delta \subseteq  \Gamma \cup \{ \infty\}$ (see Example \ref{ex: submodules of K are quasiballs}), hence we take $F_1 := K$ and $\Delta_1 := \Delta$.

	For $d>1$, let $\Delta_{1}:=\left\{ \gamma\in\Gamma_{\infty}\mid\exists v\in C \; \nu_{K^d}\left(v\right)=\gamma\right\} $. Note that $\Delta_1$ is nonempty because it contains $\infty=\nu(0)$. 
	Then there is some $i\in\left[d\right]$ such that every $\gamma\in\Delta_{1}$
is the valuation of the $i$th coordinate of some element of $C$. To see this, note that for each $i \in [d]$, the set 
$$S_i := \left\{ \gamma \in \Gamma_{\infty} \mid \exists v =(v_1, \ldots, v_d) \in C \; \nu_{K^d}(v) = \nu(v_i) = \gamma \right\}$$
is upwards closed in $\Gamma_{\infty}$. Indeed, assume $v = (v_1, \ldots, v_d) \in C$, $\gamma = \nu(v_i) = \min\{ \nu(v_j) : j \in [d] \}$ and $\delta \geq \gamma$ in $\Gamma_{\infty}$. Let $\alpha \in K$ be arbitrary with $\nu(\alpha) = \delta - \gamma$, then $\alpha \in \mathcal{O}$, hence $\alpha v \in C$, and so $\nu_{K^d}(\alpha v) = \min \{ \nu(\alpha v_j) : j \in [d] \} = \nu(\alpha v_j) = \delta$.
As we also have $\Delta_1 = \bigcup_{i \in [d]} S_i$, it follows that $\Delta_1 = S_i$ for some $i \in [d]$ as wanted (and in particular $\Delta_1$ is upwards closed in $\Gamma_{\infty}$).

Without loss of generality we may assume $i=1$. Then, given any $\gamma\in\Delta_{1}$, there is some $(\alpha, y_1, \ldots, y_{d-1}) \in C$ such that $\gamma = \nu(\alpha) \leq \min \left\{ \nu(y_j) : j \in [d-1] \right\}$. Taking $x_j := \frac{y_j}{\alpha} \in \mathcal{O}$, we thus have $(\alpha, \alpha x_1, \ldots, \alpha x_{d-1}) \in C$. Hence for any $\gamma\in\Delta_{1}$, the set 
$$X_{\gamma} := \left\{ \left(x_{1}, \ldots ,x_{d-1}\right)\in \mathcal{O}^{d-1}\mid\exists\alpha\in K\;\nu\left(\alpha\right)=\gamma \, \land \, \left(\alpha,\alpha x_{1}, \ldots, \alpha x_{d-1}\right)\in C\right\} $$
is nonempty, and convex (by Lemma \ref{lem: interm conv set}). Note that for $\gamma < \delta \in \Gamma_{\infty}$ we have $X_{\gamma} \subseteq X_{\delta}$, hence $\bigcap_{\gamma \in \Delta_1} X_{\gamma} \neq \emptyset$
by Lemma \ref{lem: inters of conv is nonempty}. That is, there exists $\left(x_{1}, \ldots, x_{d-1}\right)\in \mathcal{O}^{d-1}$
such that $\forall\gamma\in\Delta_{1}\;\exists\alpha\in K\; \left( \nu(\alpha) = \gamma \, \land \, \left(\alpha,\alpha x_{1}, \ldots, \alpha x_{d-1}\right)\in C \right)$. Hence
\begin{gather}\label{eq: MainThmeq1}
\forall \alpha\in K, \; \nu\left(\alpha\right)\in\Delta_{1}\implies\left(\alpha,\alpha x_{1}, \ldots, \alpha x_{d-1}\right)\in C
\end{gather}
(since we have $\exists \beta \in K \, \nu(\beta) = \nu(\alpha) \land (\beta, \beta x_1, \ldots, \beta x_{d-1}) \in C$, so $\frac{\alpha}{\beta} \in \mathcal{O}$ and multiplying by it we get $\left(\alpha,\alpha x_{1}, \ldots, \alpha x_{d-1}\right)\in C$).

Let $F_{1}:=\left\langle \left(1,x_{1}, \ldots, x_{d-1}\right)\right\rangle $.
Let $\pi: K^d \twoheadrightarrow K^d/F_1$ be the projection map, $f: K^d/F_1 \hookrightarrow K^d$ the valuation preserving embedding given by Lemma \ref{lem: val pres emb quotient}, and $\pi' := f \circ \pi : K^d \to K^d$. Note that $K^{d}/F_{1} \cong K^{d-1}$ as a valued $K$-vector space by Lemma \ref{lem: val pres emb quotient}, and that $\widetilde{C} := \pi(C)$ is still an $\mathcal{O}$-submodule of $K^{d}/F_{1}$. By induction hypothesis there is
a full flag $\left\{ 0\right\} \subsetneq\widetilde{F}_{2}\subsetneq \ldots \subsetneq\widetilde{F}_{d}=K^{d}/F_{1}$
and upwards-closed subsets $\nu_{K^d/F_1}(\widetilde{C}) = \Delta_{2}\supseteq \ldots \supseteq \Delta_{d}$
of $\Gamma_{\infty}$ satisfying the conclusion of the theorem with respect
to $\widetilde{C}$ (the equality $\nu_{K^d/F_1}(\widetilde{C}) = \Delta_{2}$ is by Remark \ref{rem: Delta1}). Note that 
\begin{gather} \label{eq: MainThmeq2}
	\forall v \in K^d, \; \nu_{K^d}(\pi'(v)) = \nu_{K^d/F_1}(\pi(v)) \geq \nu_{K^d}(v).
\end{gather}
  In particular  we have $\Delta_1 \supseteq \Delta_2$.

 Let the subspaces $F_{2}, \ldots, F_{d}$ be the preimages of $\widetilde{F}_{2}, \ldots, \widetilde{F}_{d}$
in $K^{d}$. We let $W := f(K^d/F_1) \subseteq K^{d}$ be the image of the valuation preserving embedding
$f: K^{d}/F_{1}\hookrightarrow K^{d}$. Then we have
\begin{gather}\label{eq: MainThmeq3}
	C=\left\{ v_{1}+w\mid v_{1}\in F_{1},\;\nu_{K^d}\left(v_{1}\right)\in\Delta_{1},\;w\in C\cap W\right\}.
\end{gather}
To see this, given an arbitrary $v \in C$, let $w := \pi'(v)$ and $v_1 := v - w$. As $\pi \circ f = \id_{K^d/F_1}$ by assumption, we have $\pi(w) = \pi( \pi' (v)) = \pi(f(\pi(v))) = \pi(v)$, hence $v_1 \in F_1$.
By \eqref{eq: MainThmeq2} we have $\nu_{K^d}(w) \geq \nu_{K^d}(v)$, and thus $\nu_{K^d}(v_1) \geq \min\{\nu_{K^d}(v),\nu_{K^d}(w)\} \geq \nu_{K^d}(v)$ as well. Thus $\nu_{K^d}(v_1) \in \Delta_1$, and $v_1 \in F_1$, which together with \eqref{eq: MainThmeq1} and the definition of $F_1$ implies $v_1 \in C$; hence $w = v - v_1 \in C$ as well. The opposite inclusion is obvious.

Furthermore, applying the isomorphism $f: K^d/F_1 \to W$ to 
\begin{gather*}
	\widetilde{C} = C/F_1 = \left\{ v_2 + \ldots + v_d \mid v_i \in \widetilde{F}_i, \nu_{K^d/F_1}(v_i) \in \Delta_i \right\}
\end{gather*}
we get
\begin{gather*}
	C\cap W=\left\{ v_{2}+ \ldots + v_{d}\mid v_{i}\in F_{i}\cap W,\;\nu_{K^d}\left(v_{i}\right)\in\Delta_{i}\right\},
\end{gather*}
which together with \eqref{eq: MainThmeq3} implies
\begin{gather*}
	C=\left\{ v_{1} + \ldots + v_{d}\mid v_{i}\in F_{i},\;\nu\left(v_{i}\right)\in\Delta_{i},\;v_{i}\in W\text{ for \ensuremath{i\geq2}}\right\}.
\end{gather*}

Now $C=\left\{ v_{1}+ \ldots +v_{d}\mid v_{i}\in F_{i},\;\nu\left(v_{i}\right)\in\Delta_{i}\right\} $
follows because for any such vectors $v_{1}, \ldots, v_{d}$, the vector $v_{i}$ (for $i\geq2$)
can be moved into $W$ by subtracting an element of $F_{1}$ with
valuation in $\Delta_{1}$, and collecting the differences in with
$v_{1}$. That is, given arbitrary $v_i \in F_i$ with $\nu(v_i) \in \Delta_i$, let $w_{i}:=\pi'\left(v_{i}\right) \in W$ for
$i\geq2$, 
and let $w_{1}:=v_{1}+\left(v_{2}-\pi'\left(v_{2}\right)\right)+ \ldots +\left(v_{d}-\pi'\left(v_{d}\right)\right)$. As above, using \eqref{eq: MainThmeq2}, for each $i \geq 2$ we have $\nu_{K^d}(v_i - \pi'(v_i)) \geq \min\{ \nu_{K^d}(v_i),  \nu_{K^d}(\pi'(v_i))\} \geq \nu_{K^d}(v_i) \in \Delta_i \subseteq \Delta_1$. Hence $\nu_{K^d} \left(w_1\right) \geq \min\{v_1, v_2 - \pi'(v_2), \ldots, v_d - \pi'(v_d)\} \in \Delta_1$. We also have $\nu_{K^d}(w_i) \geq \nu_{K^d}(v_i) \in \Delta_i$ for $i \geq 2$ by \eqref{eq: MainThmeq2}. Using that $f$ is a one-sided inverse of $\pi$ as above, we also have $v_i - \pi'(v_i) \in F_1 \subseteq F_i$ for $i \geq 2$. 
It follows that $w_i \in F_i$ for all $i \in [d]$. Putting all of this together, we get
 $w_{1}+ \ldots +w_{d}=v_{1}+ \ldots +v_{d}$, $w_{i}\in F_{i}$,
$\nu\left(w_{i}\right)\in\Delta_{i}$, and $w_{i}\in W$
for $i\geq2$.
\end{proof}

\begin{remark}\label{rem: MainThmDeltaD}
	Note that as $F_d = K^d$ in Theorem \ref{thm: main Omod pres}, we have $$\Delta_{d}=\left\{ \gamma\in\Gamma_{\infty}\mid \forall v \in K^d, \  \nu\left(v\right)=\gamma\implies v\in C\right\}.$$
That is, $\Delta_{d}$ is the quasi-radius of the largest quasi-ball around $0$ 
contained in $C$.
\end{remark}

\begin{remark}\label{rem: C intersec Fj}
	Given a convex set $0 \in C \subseteq K^d$ and any $F_i,\Delta_i, i \in [d]$ satisfying the conclusion of Theorem \ref{thm: main Omod pres} with respect to it, for every $j \in [d]$ we have 
 $$C \cap F_{j} = \left\{ v_{1}+ \ldots +v_{j}\mid v_{i}\in F_{i},\;\nu\left(v_{i}\right)\in\Delta_{i} \textrm{ for all } j \in [i] \right\}.$$
Indeed, if $x \in C \cap F_{j}$, then $x = v_1 + \ldots + v_d \in F_{j}$ for some $v_i \in F_i$ with $\nu(v_i) \in \Delta_i$ for $i \in [d]$. Then, using that the $F_i$ are increasing under inclusion and $\Delta_i$ are increasing under inclusion and upwards closed, $v_{j+1} + \ldots + v_d \in F_{j}$ and taking $v'_{j} := v_{j} + \ldots + v_d$ we have $v'_{j} \in F_{j}, \nu(v'_{j}) \geq \min \left\{ \nu(v_{i}) : j \leq i \leq d \right\} \in \Delta_{j}$ and $x = v_1 + \ldots + v_{j-1} + v'_{j}$. Conversely, any element $x = v_1 + \ldots + v_{j}$ with $v_i \in F_i, \nu(v_i) \in \Delta_i$ for $i \in [j]$ can be written as $x = v_1 + \ldots + v_d$ with $v_i := 0 \in F_i$ and $\nu(v_i) = \infty \in \Delta_i$ for $j+1 \leq i \leq d$. So $x \in C \cap F_{j}$.
\end{remark}

\begin{remark}\label{rem: choosing Fd-1}
\begin{enumerate}	\item It follows from the conclusion of Theorem \ref{thm: main Omod pres} that the subspace $F_{d-1}$ is a linear hyperplane in $K^d$, and 
every element of $C$ differs from an element of $F_{d-1}$ (and hence of $F_{d-1} \cap C$ in view of Remark \ref{rem: C intersec Fj}) by a vector in $K^d$ with
valuation in $\Delta_{d}$ (with $\Delta_{d}$ as in Remark \ref{rem: MainThmDeltaD}).
\item Conversely, $F_{d-1}$ can be chosen to be \emph{any} linear hyperplane $H$ in $K^d$
such that every element of $C$ differs from an element of $H$ by a vector in $K^d$ with
valuation in $\Delta_{d}$. To see this, let $H$ be such
a hyperplane in $K^d$. Then $C\cap H$ is a convex subset of $H \cong K^{d-1}$ containing $0$, hence an $\mathcal{O}$-submodule of $H$ by Proposition \ref{prop: conv iff subm}. Applying Theorem \ref{thm: main Omod pres} to $C \cap H$ in $H$ (with the induced valuation on $H$), there are $\Delta_{1}\supseteq\Delta_{2}\supseteq \ldots \supseteq\Delta_{d-1}$
and a full flag $\left\{ 0\right\} \subsetneq F_{1}\subsetneq\ldots\subsetneq F_{d-1}=H$,
such that $C\cap H=\left\{ v_{1}+ \ldots +v_{d-1}\mid v_{i}\in F_{i},\;\nu\left(v_{i}\right)\in\Delta_{i}\right\} $.
Then 
\begin{gather*}
	\left\{ v_{1}+ \ldots+ v_{d}\mid v_{i}\in F_{i},\;\nu\left(v_{i}\right)\in\Delta_{i}\right\} =\left\{ w+v_{d}\mid w\in C\cap H,\;\nu\left(v_{d}\right)\in\Delta_{d}\right\} =C.
\end{gather*}
\end{enumerate}

\end{remark}

\begin{expl}\label{ex: sph comple necc}
	The assumption of spherical completeness of $K$ is necessary in Theorem \ref{thm: main Omod pres}. For example, let $K:=\bigcup_{n\geq1}k\left(\left(t^{\frac{1}{n}}\right)\right)$
be the field of Puiseux series over a field $k$, and let $\widetilde{K}:=k\left[\left[t^{\mathbb{Q}}\right]\right]$
be the field of Hahn series over $k$ with rational exponents, it is the
spherical completion of $K$ (both fields have value group $\mathbb{Q}$
and valuation $\nu\left(x\right)=q$ where $x$ has leading term $t^{q}$; see e.g.~\cite[Example 3.3.23]{aschenbrenner2017asymptotic}). In particular $\sum_{n\geq1}t^{1-\frac{1}{n}} \in \widetilde{K} \setminus K$, and 
let
$$\widetilde{C}:=\left\{ \alpha\left(1,\sum_{n\geq1}t^{1-\frac{1}{n}}\right)+v\mid \alpha \in \widetilde{K}, v \in \widetilde{K}^2,  \nu_{\widetilde{K}}\left(\alpha\right)\geq0,\,\nu_{\widetilde{K}^2}\left(v\right)\geq1\right\} \subseteq\widetilde{K}^{2},$$
and let $C:=\widetilde{C}\cap K^{2}$. Then $\widetilde{C}$ is convex in $\widetilde{K}^2$, and hence $C$
is also convex as a subset of $K^{2}$. The basic idea behind why
$C$ is not of the form described in Theorem \ref{thm: main Omod pres} is that $C$ is close
enough to $\widetilde{C}$, and the subspace $F_{1}$ appearing in the conclusion of Theorem \ref{thm: main Omod pres}
for $\widetilde{C}$ must be close to $\left\langle \left(1,\sum_{n\geq1}t^{1-\frac{1}{n}}\right)\right\rangle $;
specifically, it must be $\left\langle \left(1,x+\sum_{n\geq1}t^{1-\frac{1}{n}}\right)\right\rangle $
for some $x \in K^2$ with $\nu\left(x\right)\geq1$, but $K^{2}$ contains
no such subspaces.

Indeed, by Remark \ref{rem: Delta1}, given any $F_i, \Delta_i$ satisfying the conclusion of Theorem \ref{thm: main Omod pres} with respect to $C$, the valuation of every element of $C$ must also be
the valuation of some element of $F_{1}\cap C$. So, to show that $C$
is not of the form described in Theorem \ref{thm: main Omod pres}, it suffices to show that
$C$ contains elements of valuation arbitrarily close to $0$, but
that for every $1$-dimensional subspace $F_{1}\subset K^{2}$, there
is some $q>0$ in $\Gamma$ such that every element of $F_{1}\cap C$ has valuation
at least $q$ (and note that from the definition of $C$, every element in it has positive valuation).

\begin{claim}
	For every $n\in\mathbb{N}_{\geq 1}$, there is some $v\in C$ with
$\nu_{K^2}\left(v\right)=\frac{1}{n}$. 
\end{claim}
\begin{proof}
	To see this, note that 
	\begin{gather*}
		t^{\frac{1}{n}}\left(1,\sum_{m=1}^{n-1}t^{1-\frac{1}{m}}\right)=t^{\frac{1}{n}}\left(1,\sum_{m\geq1}t^{1-\frac{1}{m}}\right)-t^{\frac{1}{n}}\left(0,\sum_{m\geq n}t^{1-\frac{1}{m}}\right)\in C
	\end{gather*}

as $\nu_{K}\left(t^{\frac{1}{n}}\right)=\frac{1}{n}\geq0$ and $\nu_{K^2}\left(t^{\frac{1}{n}}\left(0,\sum_{m\geq n}t^{1-\frac{1}{m}}\right)\right)=\frac{1}{n}+\left(1-\frac{1}{n}\right)\geq1$.
\end{proof}

\begin{claim}
	For every $1$-dimensional subspace $F_{1}\subset K^{2}$,
there is some $n\in\mathbb{N}_{n \geq 1}$ such that every element of $F_{1}\cap C$
has valuation at least $\frac{1}{n}$.
\end{claim}
\begin{proof}
	We prove this by breaking into
two cases.

\noindent \emph{Case 1.} $F_{1}=\left\langle \left(0,1\right)\right\rangle $. 

Assume $x \in F_1 \cap C$, then $x = (x_1,x_2) = \alpha\left(1,\sum_{n\geq1}t^{1-\frac{1}{n}}\right)+v$
for some $\alpha  \in K, v = (v_1, v_2) \in \widetilde{K}^2$ with $\nu_{\widetilde{K}}(\alpha) \geq 0, \nu_{\widetilde{K}^2}(v) \geq 1$, and $x_1 = 0$, so $\alpha = - v_1$. But $1 \leq \nu_{\widetilde{K}^2}(v) = \min \{\nu_{\widetilde{K}}(v_1), \nu_{\widetilde{K}}(v_2) \}$, hence $\nu_{\widetilde{K}}(\alpha) \geq 1$ as well. Since $\nu_{\widetilde{K}} \left( \sum_{n\geq1}t^{1-\frac{1}{n}} \right) = 0$, it follows that $\nu_{\widetilde{K}^2}(x) = \min \left\{ \nu_{\widetilde{K}}(0), \nu_{\widetilde{K}} \left(\alpha\left( \sum_{n\geq1}t^{1-\frac{1}{n}} \right) \right)\right\} \geq 1$. Thus every element of
$F_{1}\cap C$ has valuation at least $1$.

\noindent \emph{Case 2.} $F_{1}=\left\langle \left(1,x\right)\right\rangle $ for some
$x\in K$. 

Given any $x \in K$, there must exist some $n\in\mathbb{N}$ such that $\nu_{\widetilde{K}}\left(x-\sum_{m\geq1}t^{1-\frac{1}{m}}\right)\leq1-\frac{1}{n}$.
Given any $v\in F_{1}\cap C$, we have 
$$v = \alpha\left(1,x\right) = \beta\left(1,\sum_{m\geq1}t^{1-\frac{1}{m}}\right)+w$$
 for some $\alpha\in K$,
some $\beta\in\widetilde{K}$ with $\nu_{\widetilde{K}}\left(\beta\right)\geq0$ and $w=\left(w_{1},w_{2}\right)\in\widetilde{K}^{2}$
with $\nu_{\widetilde{K}^2}\left(w\right)\geq1$.
Without loss of generality $\alpha \neq 0$, so we have
\begin{gather*}
	x=\frac{\alpha x}{\alpha}=\left(w_{2}+\beta\sum_{m\geq1}t^{1-\frac{1}{m}}\right)\left(w_{1}+\beta\right)^{-1}=\left(\frac{w_{2}}{\beta}+\sum_{m\geq1}t^{1-\frac{1}{m}}\right)\left(1+\frac{w_{1}}{\beta}\right)^{-1}.
\end{gather*}
If $\nu_{\widetilde{K}}\left(\beta\right)<\frac{1}{n}$, then 
\begin{gather*}
\nu_{\widetilde{K}}\left(\frac{w_{1}}{\beta}\right) > 1-\frac{1}{n}, \  \nu_{\widetilde{K}}\left(\frac{w_{2}}{\beta}\right)>1-\frac{1}{n}, \  
\nu_{\widetilde{K}}\left(\left(1+\frac{w_{1}}{\beta}\right)^{-1}\right)=0 \textrm{, and}\\
\nu_{\widetilde{K}}\left(\left(1+\frac{w_{1}}{\beta}\right)^{-1}-1\right)>1-\frac{1}{n} \textrm{, so}\\
\nu\left(x-\sum_{m\geq1}t^{1-\frac{1}{m}}\right)=\nu\left(\frac{w_{2}}{\beta}\left(w_{1}+\beta\right)^{-1}+\left(\sum_{m\geq1}t^{1-\frac{1}{m}}\right)\left(\left(1+\frac{w_{1}}{\beta}\right)^{-1}-1\right)\right)\\>1-\frac{1}{n},	
\end{gather*}
a contradiction to the choice of $n$. Thus $\nu\left(\beta\right)\geq\frac{1}{n}$, and
hence $\nu\left(v\right)\geq\frac{1}{n}$.
\end{proof}

Thus no $1$-dimensional subspace $F_{1}$ of $K^2$ can fill its desired role
in the presentation for $C$. 
\end{expl}

Theorem \ref{thm: main Omod pres} implies the following simple description of convex sets over spherically complete valued fields.
\begin{cor}\label{cor: descr of conv sets in sph compl}
	If $K$ is a spherically complete valued field and $d \in \mathbb{N}_{\geq 1}$, then the non-empty convex subsets
of $K^{d}$ are precisely the affine images of $\nu^{-1}\left(\Delta_{1}\right)\times \ldots \times\nu^{-1}\left(\Delta_{d}\right)$
for some upwards closed $\Delta_{1}, \ldots, \Delta_{d} \subseteq \Gamma_{\infty}$.
\end{cor}
\begin{proof}
Let $C \subseteq K^d$ be an affine image of $\nu^{-1}\left(\Delta_{1}\right)\times \ldots \times\nu^{-1}\left(\Delta_{d}\right)$
for some upwards closed $\Delta_{1}, \ldots, \Delta_{d} \subseteq\Gamma_{\infty}$. Note that $\nu^{-1}\left(\Delta_{1}\right)\times \ldots \times\nu^{-1}\left(\Delta_{d}\right)$ is convex, and an image of a convex set under an affine map is convex (Example \ref{ex: some conv sets}), hence $C$ is convex.

	Conversely, let $\emptyset \neq C \subseteq K^d$ be convex. Since the affine images of
$\mathcal{O}$-submodules of $K^d$ give us all non-empty convex sets by Proposition \ref{prop: conv iff subm}, without loss of generality $0\in C$ and $C$ is an $\mathcal{O}$-submodule of $K^d$. 
Let $\left\{ 0\right\} \subsetneq F_{1}\subsetneq \ldots \subsetneq F_{d}=K^{d}$
and $\nu_{K^d}(C) = \Delta_{1}\supseteq\Delta_{2}\supseteq \ldots \supseteq\Delta_{d}$ 
be as given by Theorem \ref{thm: main Omod pres} for $C$. Using Lemma \ref{lem: val pres emb quotient} we can choose
$v_{1}, \ldots ,v_{d}\in K^{d}$ such that for every $i \in [d]$ we have:
\begin{enumerate}
	\item $v_{1},\ldots,v_{i}$ is a basis for $F_{i}$,
	\item $\nu\left(v_{i}\right)=0$, 
	\item $\nu\left(v_{i}+x\right)\leq0$ for all $x\in F_{i-1}$.
\end{enumerate}
  Then $C$ is the image of $\nu^{-1}\left(\Delta_{1}\right)\times \ldots \times\nu^{-1}\left(\Delta_{d}\right)$
under the linear map $f:K^{d}\rightarrow K^{d}$ such that $f\left(e_{i}\right)=v_{i}$,
where $e_{i}$ is the $i$th standard basis vector. Indeed, if $x \in f \left(\nu^{-1}\left(\Delta_{1}\right)\times \ldots \times\nu^{-1}\left(\Delta_{d}\right) \right)$ then $x = \sum_{i=1}^{d} c_i v_i$ for some $c_i$ with $\nu(c_i) \in \Delta_i$. Using (2)  this implies $\nu(c_i v_i) = \nu(c_i) \in \Delta_i $, and $c_i v_i \in F_i$, hence $x \in C$.  Conversely, let $x$ be an arbitrary element of $C$, then $x = w_1 + \ldots + w_d$ for some $w_i \in F_i$ with $\nu(w_i) \in \Delta_i$. Each $w_i$ is a linear combination of $v_1, \ldots, v_i$, say $w_i = \sum_{j=1}^{i} c_{i,j}v_j$. 

Now we claim that for any $i \in [d]$, $\alpha \in K$ and $v \in F_{i-1}$ we have $\nu(\alpha v_i + v) = \min \{ \nu(\alpha v_i), \nu(v) \}$. Indeed, replacing $v$ and $\alpha$ by $\alpha^{-1}v \in F_{i-1}$ and $\alpha^{-1}\alpha \in K$, respectively, changes both sides of the claimed equality by the same amount, hence we may assume that $\alpha = 0$ or $\alpha = 1$. The first case holds trivially, in the second case we need to show that $\nu(v_i + v) = \min \{ \nu(v_i), \nu(v) \}$. If $\nu(v_i) \neq \nu(v)$ this holds by the ultrametric inequality, so we assume $\nu(v_i) = \nu(v) = 0$ (using (2)). Then, using (3), $0 \geq \nu(v_i + v) \geq \min \{ \nu(v_i), \nu(v)  \} = 0$, so $\nu(v_i+v) = 0$ as well.

Applying this claim by induction on $i \in [d]$, we get
\[ \nu \left( \sum_{j=1}^{i} c_{i,j} v_j \right) = \min_{j} \left\{ \nu(c_{i,j} v_j) \right\},\]
 which using (2) implies $\nu(w_i) = \nu \left( \sum_{j=1}^{i} c_{i,j} v_j \right) = \min_{j} \left\{ \nu(c_{i,j} ) \right\}$ for each $i \in [d]$. As for each $i \in [d]$ we have $\nu(w_i) \in \Delta_i$ and $\Delta_i$ is upwards closed, it follows that $\nu(c_{i,j}) \in \Delta_i$ for all $i \in [d], j \in [i]$. Regrouping the summands $c_{i,j}v_i$, it follows that $x = w_1 + \ldots + w_d$ is a linear combination of $v_1, \ldots, v_d$ where the coefficient of $v_i$ has valuation in $\Delta_i$, hence $x$ belongs to $f \left(\nu^{-1}\left(\Delta_{1}\right)\times \ldots \times\nu^{-1}\left(\Delta_{d}\right) \right)$.
\end{proof}

We can eliminate the assumption of spherical completeness of the field when only
considering convex hulls of finite sets. 
We will say that a convex set is
\emph{finitely generated} if it is the convex hull of a finite set of points.

\begin{lemma}\label{lem: fin gen}
	A subset $C\subseteq K^{d}$ is a finitely generated
$\mathcal{O}$-module if and only if it is a finitely generated convex set and contains
$0$.
\end{lemma}
\begin{proof}
	If an $\mathcal{O}$-module $C\subseteq K^{d}$ is generated as
an $\mathcal{O}$-module by some finite set $X$, then it is the convex
hull of $X\cup\left\{ 0\right\} $. If a set $C$ is the convex hull of some
finite set $X$ and contains $0$, then it is an $\mathcal{O}$-module by Proposition \ref{prop: conv iff subm}, clearly generated as an $\mathcal{O}$-module
by $X$.
\end{proof}

We have the following analog of Theorem \ref{thm: main Omod pres} in the finitely generated case over an arbitrary valued field.

\begin{cor}\label{cor: class of fg modules}
	Let $K$ be an \emph{arbitrary} valued field and $C$ a \emph{finitely generated} convex set containing
$0$. Then there is a full flag $\left\{ 0\right\} \subsetneq F_{1}\subsetneq \ldots \subsetneq F_{d}=K^{d}$
and an increasing sequence $\gamma_{1}\leq\gamma_{2}\leq \ldots \leq\gamma_{d}\in\Gamma_{\infty}$
such that 
\[C=\left\{ v_{1}+ \ldots +v_{d}\mid v_{i}\in F_{i},\;\nu\left(v_{i}\right)\geq\gamma_{i}\right\}.\]
\end{cor}
\begin{proof}
	Let $C \ni 0$ be the convex hull of some finite set $X \subseteq K^d$. By a repeated application of Proposition \ref{prop: Radon}, $C$ is the convex
hull of some $d+1$ elements $v_{0}, \ldots, v_{d}$ from $X$ (possibly with  $x_i = x_j$ for some $i,j$). As $0 \in C$, we have $0 = \sum_{i=0}^{d} \alpha_i v_i$ for some $\alpha_i \in \mathcal{O}$ with $\sum_{i=0}^{d} \alpha_i = 1$. Let $j$ be such that $\nu(\alpha_{j})$ is minimal among $\{ \nu(\alpha_i): 0 \leq i \leq d \}$. In particular $\alpha_j \neq 0$, hence $v_j = \left(1 - \sum_{i \neq j} \frac{\alpha_i}{\alpha_j} \right)0 + \sum_{i \neq j}\frac{\alpha_i}{\alpha_j}v_i$. By the choice of $j$ we have $\frac{\alpha_i}{\alpha_j} \in \mathcal{O}$ for all $i \neq j$, hence also $1 - \sum_{i \neq j} \frac{\alpha_i}{\alpha_j} \in \mathcal{O} $, thus $v_j \in \conv \left( \{0\} \cup \{ v_i : i \neq j\} \right)$, and so also $C = \conv \left( \{0\} \cup \{ v_i : i \neq j\} \right)$. Reordering if necessary, we can thus assume that $C$ is the convex hull of some $\left\{ 0,v_{1},\ldots,v_{d}\right\} \subseteq C $ with $\nu\left(v_{1}\right)\leq\nu\left(v_{i}\right)$ for each $i \in [d]$.

Let $F_{1}:=\left\langle v_{1}\right\rangle $ and $\gamma_{1}:=\nu\left(v_{1}\right)$.
Let $\pi_1: K^d \twoheadrightarrow K^d/F_1$ be the projection map, $f_1: K^d/F_1 \hookrightarrow K^d$ the valuation preserving embedding given by Lemma \ref{lem: val pres emb quotient}, $V_1 := f_1 \left(K^{d}/F_{1} \right)$ and $\pi'_1 := f_1 \circ \pi_1 : K^d \to K^d$. 

For $i \geq 2$, as explained after \eqref{eq: MainThmeq3} in the proof of Theorem \ref{thm: main Omod pres} we have $v_i - \pi'_1(v_i) \in F_1$; and by  \eqref{eq: MainThmeq2} there and assumption we have $\nu(\pi'_1(v_i)) \geq \nu(v_i) \geq \nu(v_1)$. So $v_i - \pi'_1(v_i) \in \mathcal{O} v_1$ for all $i \geq 2$, which implies 
$$\conv\left( \left\{ 0, v_{1},\pi'_1\left(v_{2}\right), \ldots,\pi'_1\left(v_{d}\right) \right\}\right)=\conv\left( \left\{ 0, v_{1}, \ldots, v_{d} \right\}\right) = C.$$
Without loss of generality we suppose $\nu\left(\pi'_1\left(v_{2}\right)\right)\leq\nu\left(\pi'_1\left(v_{i}\right)\right)$
for $i\geq3$, and let $F_{2}:=\left\langle v_{1},\pi'_1(v_{2})\right\rangle $
and $\gamma_{2}:=\nu\left(\pi'_1\left(v_{2}\right)\right) \geq \nu(v_1) = \gamma_1$ by assumption.
By definition of the valuation on the quotient space, using the properties of $f$, we have 
$$\nu_{K}(\pi'_1(v_i)) = \nu_{K^d/F_1}(\pi_1(v_i)) = \nu_{K^d/F_1}(\pi_1( \pi'_1(v_i))) \geq \nu_{K^d}(\pi'_1(v_i) + \alpha v_1)$$
 for all $\alpha \in K$.
As in the proof of Corollary \ref{cor: descr of conv sets in sph compl}, this implies $\nu(\beta \pi_1'(v_i) + \alpha v_1) = \min\{ \beta \nu(\pi_1'(v_i)), \nu(\alpha v_1)) \}$ for all $i \geq 2$ and $\alpha, \beta \in K$. It follows that
 $$\left\{ nv_{1}+m \pi'_1(v_{2})\mid n,m\in\mathcal{O}\right\} =\left\{ w_{1}+w_{2}\mid w_{i}\in F_{i},\;\nu\left(w_{i}\right)\geq\gamma_{i}\right\}. $$
To see that the set on the right is contained in the set on the left, assume $x = w_1 + w_2$ for some $w_i \in F_i, \nu(w_i) \geq \gamma_i$. Then $w_1 = \alpha_1 v_1$ and $w_2 = \alpha_2 v_1 + \beta \pi'_1(v_2)$ for some $\alpha_1, \alpha_2, \beta \in K$, and by the observation above $\gamma_2 \leq \nu(w_2) = \min\{ \nu( \alpha_2 v_1) , \nu(\beta \pi'_1(v_2)) \} $. So $x = (\alpha_1 + \alpha_2) v_1 + \beta \pi'_1(v_2)$, $\nu((\alpha_1 + \alpha_2)v_1) \geq \gamma_1 = \nu(v_1)$, so $(\alpha_1 + \alpha_2) \in \mathcal{O}$, and $\nu(\beta) \geq \gamma_2$, as wanted.

Now we replace $v_i$ by $\pi'_1(v_i)$ for $i \geq 2$, and let 
$\pi_2: K^d \twoheadrightarrow K^d/F_2$ be the projection map, $f_2: K^d/F_2 \hookrightarrow K^d$ the valuation preserving embedding given by Lemma \ref{lem: val pres emb quotient}, $V_2 := f_2 \left(K^{d}/F_{2} \right)$ and $\pi'_2 := f_2 \circ \pi_2 : K^d \to K^d$. 
For $i\geq3$, $v_{i}-\pi'_{2}\left(v_{i}\right)\in F_{2}$ and $v_{i}-\pi'_{2}\left(v_{i}\right)\in\mathcal{O}v_{1}+\mathcal{O}v_{2}$,
so again replacing $v_{i}$ with $\pi'_{2}\left(v_{i}\right)$ for $i \geq 3$ does
not change the convex hull. Again we may assume $\nu\left(\pi'_{2}\left(v_{3}\right)\right)\leq\nu\left(\pi'_{2}\left(v_{i}\right)\right)$
for $i\geq4$, and let $F_{3}:=\left\langle v_{1},v_{2},v_{3}\right\rangle $
and $\gamma_{3}:=\nu\left(\pi'_{2}\left(v_{3}\right)\right)$.
Repeating
this argument as above $d$ times, we have chosen vectors $v_i$, increasing spaces $F_i = \langle v_1, \ldots, v_i \rangle$ and increasing $\gamma_i = \nu(v_i) \in \Gamma$ for $i \in [d]$ so that 
\begin{gather*}
	C = \conv \left( \left\{ 0, v_1, \ldots, v_d  \right\}\right) = \\
	\left\{ n_1v_{1}+ \ldots + n_d v_d\mid n_i\in\mathcal{O}\right\} =\left\{ w_{1}+ \ldots + w_{d}\mid w_{i}\in F_{i},\;\nu\left(w_{i}\right)\geq\gamma_{i}\right\}. \qedhere
\end{gather*}
\end{proof}

\section{Combinatorial properties of convex sets}\label{sec: combinatorics of conv sets}

The following definition is from \cite[Section 2.4]{aschenbrenner2016vapnik}.

\begin{defn}
	Given a set $X$ and $d \in \mathbb{N}_{\geq 1}$, a family of subsets $\mathcal{F} \subseteq \mathcal{P}\left(X\right)$
has \emph{breadth} $d$ if any nonempty intersection of finitely many sets
in $\mathcal{F}$ is the intersection of at most $d$ of them, and $d$
is minimal with this property.
\end{defn}

\begin{lemma}\label{lem: field ext conv}
	Let $K$ be a valued field and $S$ a convex subset of $K^d$.
	\begin{enumerate}
		\item If $0 \in S$ and $S$ is finitely generated, then it is generated as an $\mathcal{O}$-module by a finite linearly independent set of vectors.
		\item Let $\widetilde{K}$ be a valued field extension of $K$ and $\widetilde{S} := \conv_{\widetilde{K}^d}(S) \subseteq \widetilde{K}^d$. Then $\widetilde{S} \cap K^d = S$.
	\end{enumerate}
\end{lemma}
\begin{proof}
	(1) By Lemma \ref{lem: fin gen}, $S$ is generated as an $\mathcal{O}$-module by some finite set $v_1, \ldots, v_n \in S$. Assume these vectors are not linearly independent, then $0 = \sum_{i \in [n]} \alpha_i v_i$ for some $\alpha_i \in K$ not all $0$. Let $i \in [n]$ be such that $\nu(\alpha_i) \leq \nu(\alpha_j)$ for all $j \in [n]$, in particular $\alpha_i \neq 0$. Then $v_i = \sum_{j \neq i} \frac{\alpha_j}{-\alpha_i}v_j$ and $\nu \left( \frac{\alpha_j}{-\alpha_i} \right) = \nu(\alpha_j) - \nu(\alpha_i) \geq 0$, hence $\frac{\alpha_j}{-\alpha_i} \in \mathcal{O}$ for all $j \neq i$, and  $S$ is still generated as an $\mathcal{O}$-module by the set $\{ v_j : j \neq i \}$. Repeating this finitely many times, we arrive at a linearly independent set of generators.
	
	(2) Since convexity is invariant under translates, we may assume $0 \in S$. Since every element in the convex hull of a set is in the convex
hull of some finite subset, we may also assume that $S$ is finitely generated as an $\mathcal{O}$-module, and by (1) let $v_1, \ldots, v_n \in S$ be a linearly independent (in the vector space $K^d$, so $n \leq d$) set of its generators. Let $v_{n+1}, \ldots, v_{d} \in K^d$ be so that $\{v_i : i \in [d] \}$ is a basis of $K^d$, and say $v_i = (v_{i,j} : j \in [d])$ with $v_{i,j} \in K$. Then the square matrix $A := (v_{i,j} : i,j \in [d]) \in  M_{d \times d}(K)$ is invertible, so $A^{-1} \in M_{d \times d}(K) \subseteq M_{d\times d}(\widetilde{K})$, so $A$ is also invertible in $M_{d\times d} (\widetilde{K})$, hence $\{v_i : i \in [d]\}$ are linearly independent vectors in $\widetilde{K}^d$ as well.
But now if $\sum_{i \in [n]} \alpha_i v_i = u$ for some $\alpha_i \in \widetilde{K}$ and $u \in K^d$, then necessarily $\alpha_i \in K$ for all $i$ (otherwise we would get a non-trivial linear combination of $v_1, \ldots, v_d$ in $\widetilde{K}^d$). In particular, any element of the $\mathcal{O}_{\widetilde{K}}$-module generated by $v_1, \ldots, v_n$ which is in $K^d$ already belongs to the $\mathcal{O}_{K}$-module generated by $v_1, \ldots, v_n$, hence $\widetilde{S} \cap K^d = S$.
\end{proof}

We can now demonstrate an (optimal) finite bound on the breadth of the family of convex sets over valued fields. In sharp contrast, over the reals there is no finite bound on the breadth already for convex subsets of $\mathbb{R}^2$ (for any $n$, a convex $n$-gon in $\mathbb{R}^2$ is the intersection of $n$ half-planes, but not the intersection of any fewer of them).
\begin{theorem}\label{thm: breadth d}
Let $K$ be a valued field and $d \geq 1$. Then the family $\Conv_{K^{d}}$ has breadth $d$. That is, any
nonempty intersection of finitely many convex subsets of $K^{d}$
is the intersection of at most $d$ of them.	
\end{theorem}
\begin{proof}

The family $\Conv_{K^{d}}$ cannot have breadth less than $d$
because the $d$ coordinate-aligned hyperplanes are convex, have common intersection
$\left\{ 0\right\} $, but any $d-1$ of them intersect in a line. 

We now show that $\Conv_{K^{d}}$ has breadth at most $d$, by induction on $d$. The case $d=1$ is clear by Example \ref{ex: some conv sets}(1) since for any two quasi-balls, they are either disjoint or one is contained in the other. For $d > 1$, assume $C_1, \ldots, C_{n} \in \Conv_{K^d}$ with $n \geq d$ are convex and satisfy $\bigcap_{i \in [n]}C_i \neq \emptyset$. Translating, we may assume $0 \in \bigcap_{i \in [n]}C_i$.

We may also assume that $K$ is spherically complete.  Indeed, let $\widetilde{K}$ be a spherical completion of $K$ as in Fact \ref{fac: sph compl exists}, and let $\widetilde{C}_i := \conv_{\widetilde{K}^d}(C_i) \in \Conv_{\widetilde{K}^d}$. By Lemma \ref{lem: field ext conv}(2), $\widetilde{C}_i \cap K^d = C_i$ for each $i \in [n]$. Hence $\bigcap_{i \in [n]}\widetilde{C}_i \neq \emptyset$, and if $\bigcap_{i \in [n]}\widetilde{C}_i = \bigcap_{i \in S}\widetilde{C}_i$ for some $S \subseteq [n]$ with $|S| \leq d$, then also $\bigcap_{i \in [n]}C_i = \bigcap_{i \in S}C_i$.

Then let the vector subspaces $\left\{ 0\right\} \subsetneq F_{1}\subsetneq \ldots \subsetneq F_{d}=K^{d}$
and the upwards closed subsets $\Delta_{1}\supseteq\Delta_{2}\supseteq \ldots \supseteq\Delta_{d}$  of $\Gamma_{\infty}$
be as given by Theorem \ref{thm: main Omod pres} for the convex set $C := C_{1}\cap \ldots \cap C_{n}$.
By Remark \ref{rem: MainThmDeltaD} we have 
$$\Delta_{d}=\left\{ \gamma\in\Gamma_{\infty}\mid \forall v \in K^d, \  \nu\left(v\right)=\gamma\implies v\in C_{1}\cap \ldots \cap C_{n}\right\}. $$
It follows that there is some $i_{d}\in\left[n\right]$ such that in fact
\begin{gather}\label{eq: breadth 1}
 \Delta_{d}=\left\{ \gamma\in\Gamma_{\infty}\mid\ \forall v \in K^d, \  \nu\left(v\right)=\gamma\implies v\in C_{i_{d}}\right\} 
\end{gather}
(since these are finitely many upwards closed sets in $\Gamma$, their intersection is already given by one of them).

Let $\left\{ 0\right\} \subsetneq F'_{1}\subsetneq \ldots \subsetneq F'_{d}=K^{d}$
and $\Delta'_{1}\supseteq\Delta_{2}'\supseteq \ldots \supseteq\Delta'_{d}$ 
be as given by Theorem \ref{thm: main Omod pres} for 
$C_{i_d}$. By Remark \ref{rem: choosing Fd-1}(1), $F'_{d-1}$ is a linear hyperplane so that every element of $C_{i_d}$ differs from an element of $F'_{d-1} \cap C_{i_d}$ by a vector with valuation in $\Delta_{d}'$. As $\Delta_d = \Delta'_d$ by \eqref{eq: breadth 1} and $C \subseteq C_{i_d}$,  by Remark \ref{rem: choosing Fd-1}(1) we may assume that $F_{d-1} = F'_{d-1}$, hence every element in $C_{i_d}$
differs from an element of $F_{d-1} \cap C_{i_d}$ by a vector with valuation in $\Delta_{d}$.

Consider $C \cap F_{d-1} = C_{1}\cap \ldots \cap C_{n}\cap F_{d-1}=\left(C_{1}\cap F_{d-1}\right)\cap \ldots \cap\left(C_{n}\cap F_{d-1}\right)$. Note that each $C_i \cap F_{d-1}$ is a convex subset of $F_{d-1} \cong K^{d-1}$, so by induction hypothesis there exist $i_{1}, \ldots, i_{d-1}\in\left[n\right]$
such that 
\begin{gather}\label{eq: breadth 2}
	C_{i_{1}}\cap \ldots \cap C_{i_{d-1}}\cap F_{d-1}=C_{1}\cap \ldots \cap C_{n}\cap F_{d-1}= C \cap F_{d-1}. 
\end{gather}

Let $x \in C_{i_1} \cap \ldots \cap C_{i_d}$ be arbitrary. As $x \in C_{i_d}$, by the choice of $F_{d-1}$, $x = w + v_d$ for some $w \in F_{d-1}$ and $v_d \in K^d$ with $\nu(v_d) \in \Delta_d$. 
By the choice of $\Delta_d$ we have in particular $v_d \in C_{i_1} \cap \ldots \cap C_{i_d}$. And as each $C_i$ is a module, it follows that also $w \in C_{i_1} \cap \ldots \cap C_{i_d}$.
Combining this with \eqref{eq: breadth 2} and using Remark \ref{rem: C intersec Fj} (for $j = d-1$)
 we thus have
\begin{gather*}
	C_{i_{1}}\cap \ldots \cap C_{i_{d}}=\left\{ w+v_{d}\mid w\in C_{i_{1}}\cap \ldots \cap C_{i_{d}}\cap F_{d-1},\;\nu\left(v_{d}\right)\in\Delta_{d}\right\} =\\
	\left\{ w+v_{d}\mid w\in C \cap F_{d-1},\;\nu\left(v_{d}\right)\in\Delta_{d}\right\} =\\
	\left\{ v_{1}+ \ldots +v_{d}\mid v_{i}\in F_{i},\;\nu\left(v_{i}\right)\in\Delta_{i}\right\}=\\
	C_{1}\cap \ldots \cap C_{n}.\qedhere
\end{gather*}
\end{proof}

\begin{defn}\label{def: Helly num}
\begin{enumerate}
	\item A family of sets $\mathcal{F} \subseteq \mathcal{P}\left(X\right)$ has \emph{Helly number}
$k \in \mathbb{N}_{\geq 1}$ if given any $n \in \mathbb{N}$ and any sets $S_{1}, \ldots, S_{n}\in \mathcal{F}$, if every
$k$-subset of $\left\{ S_{1}, \ldots, S_{n}\right\} $ has nonempty intersection,
then $\bigcap_{i \in [n]}S_{i}\neq\emptyset$.
\item The \emph{Helly number
of $\mathcal{F}$} refers to the minimal $k$ with this property (or $\infty$ if it does not exist).
 \item We say
that $\mathcal{F}$ has the \emph{Helly property} if it has a finite Helly number.
\end{enumerate}
\end{defn}

\begin{theorem}\label{thm: Helly number}
Let $K$ be a valued field and $d \geq 1$. Then the  Helly number of $\Conv_{K^{d}}$ is $d+1$.	
\end{theorem}
\begin{proof}
The Helly number is bounded by the Radon number minus $1$ in an arbitrary convexity space (see Section \ref{sec: convexity spaces}), but we include a proof for completeness. 
	Let $n$ be arbitrary, and  let $S_{1}, \ldots, S_{n}\subseteq K^{d}$ be convex sets so that any $d+1$ of
them have a non-empty intersection. We will show by induction on $n$ that $S_{1}\cap \ldots \cap S_{n}\neq\emptyset$.

\noindent \emph{Base case: $n=d+2$.} 

\noindent By assumption for each $i \in [d+2]$ there exists some $x_i \in K^d$ so that 
$x_{i}\in\bigcap_{j \in [d+2] \setminus \{i\}}S_{j}$.
By Proposition \ref{prop: Radon} there exists some $i^* \in [d+2]$ so that $x_{i^*} \in \conv\left(\left\{ x_{i}\mid i \neq i^*\right\} \right)$. By the choice of the $x_i$'s we have $x_{i^*} \in S_i$ for all $i \neq i^*$. We also have $x_i \in S_{i^*}$ for all $i \neq i^*$, $S_{i^*}$ is convex and $x_{i^*} \in \conv\left(\left\{ x_{i}\mid i\neq i^*\right\} \right)$, hence $x_{i^*} \in S_{i^*}$. Thus $x_{i^*} \in \bigcap_{i \in [d+2]} S_i$, as wanted.

\noindent \emph{Inductive step: $n>d+2$.} 

\noindent Let $\widetilde{S}_{n-1}:=S_{n-1}\cap S_{n}$, in particular 
$\widetilde{S}_{n-1}$ is convex.
By induction hypothesis, any $n-1$ sets from $\left\{ S_{1}, \ldots,S_{n} \right\}$ have a non-empty intersection. Hence any $n-2$ sets from $\left\{ S_{1}, \ldots ,S_{n-2},\widetilde{S}_{n-1} \right\}$
have a non-empty intersection. As $n-2\geq d+1$ by assumption, applying the induction hypothesis again we get 
$$S_{1}\cap \ldots \cap S_{n}=S_{1}\cap \ldots \cap S_{n-2}\cap\widetilde{S}_{n-1}\neq\emptyset.$$
This completes the induction, and shows that $\Conv_{K^{d}}$
has Helly number $d+1$.

It remains to show that $\Conv_{K^{d}}$
does not have Helly number $d$.
 Let  $e_{i}\in K^{d}$ be the $i$th standard basis vector. In particular the set 
$E := \left\{ 0,e_{1}, \ldots , e_{d}\right\} $ is affinely independent, hence 
the intersection of the affine spans of its $d+1$ maximal proper subsets
is empty.
The convex hull of a subset of $K^{d}$ is contained in its affine
hull, hence the intersection of the $d+1$ convex hulls of
its maximal proper subsets is also empty. 
But for any $d$ among the $(d+1)$ maximal proper subsets of $E$, some element of $E$ belongs to their intersection, and hence in particular the intersection of their convex hulls is non-empty. \end{proof}

We recall some terminology around the \emph{Vapnik-Chervonenkis dimension} (and refer to \cite[Sections 1 and 2]{aschenbrenner2016vapnik} for further details).

\begin{defn}\label{def: vc-dim and co}
Let $\mathcal{F} \subseteq \mathcal{P}(X)$ be a family of subsets of $X$.
	\begin{enumerate}
	\item For a subset $Y \subseteq X$, we let $\mathcal{F} \cap Y := \{ S \cap Y : S \in Y \} \subseteq \mathcal{P}(Y)$.
		\item We say that $\mathcal{F}$ \emph{shatters} a
subset $Y\subseteq X$ if $\mathcal{F} \cap Y = \mathcal{P}(Y)$. 

\item The \emph{VC-dimension} of $\mathcal{F}$, or $\VC(\mathcal{F})$, 
is the largest $k\in\mathbb{N}$ (if one exists) such that $\mathcal{F}$
shatters some subset of $X$ size $k$. If $\mathcal{F}$ shatters arbitrarily large
finite subsets of $X$, then it is said to have infinite VC-dimension.
\item The \emph{dual family} $\mathcal{F}^{*}\subseteq \mathcal{P}\left( \mathcal{F} \right)$
is given by $\mathcal{F}^{*} = \left\{ S_{x}\mid x\in X\right\} $, where
$S_{x}=\left\{ A\in \mathcal{F} \mid x\in A\right\} $. 

\item The \emph{dual VC-dimension} of $\mathcal{F}$, or $\VC^{*}(\mathcal{F})$, is the VC-dimension
of $\mathcal{F}^{*}$. Equivalently, it is the largest $k\in\mathbb{N}$ (or $\infty$ if no such $k$ exists) 
such that there are sets $S_{1}, \ldots, S_{k}\in \mathcal{F}$ that generate a Boolean algebra with $2^k$ atoms (i.e.~for any distinct $I,J \subseteq[k]$, $\bigcap_{i\in I}S_i \cap \bigcap_{i\in [k] \setminus I} \left( X \setminus S_i \right) \neq \bigcap_{i\in J}S_i \cap \bigcap_{i\in [k] \setminus J} \left( X \setminus S_i\right) $).
\item The \emph{shatter function} $\pi_{\mathcal{F}}: \mathbb{N} \to \mathbb{N}$ of $\mathcal{F}$ is
$$\pi_{\mathcal{F}}(n) := \max \left\{ |\mathcal{F} \cap Y| : Y \subseteq X, |Y|  = n \right\}.$$
\item By the Sauer-Shelah lemma (see e.g.~\cite[Lemma 2.1]{aschenbrenner2016vapnik}, if $\VC(\mathcal{F}) \leq d$, then $\pi_{\mathcal{F}}(n) \leq \left(\frac{e}{d}\right)^d n^d$ for all $n\geq d$ (and $\pi_{\mathcal{F}}(n) = 2^n$ for all $n$ if $\VC(\mathcal{F}) = \infty$).

\item The \emph{VC-density} of $\mathcal{F}$, or $\vc(\mathcal{F})$, is the infimum of all $r \in \mathbb{R}_{>0}$ so that $\pi_{\mathcal{F}}(n) = O(n^r)$, and $\infty$ if there is no such $r$. (In particular $\vc(\mathcal{F}) \leq \VC( \mathcal{F})$.)
\item Finally, we define the \emph{dual shatter function} $\pi^{*}_{\mathcal{F}} := \pi_{\mathcal{F}^{*}}$ and the \emph{dual $\VC$-density} $\vc^{*}(\mathcal{F}) := \vc(\mathcal{F}^{*})$ of the family $\mathcal{F}$.
	\end{enumerate}
\end{defn}

\begin{remark}\label{rem: VC dim of subfamily}
	Note that if $\mathcal{F} \subseteq \mathcal{P}(X)$ and $Y \subseteq X$, then $\VC(\mathcal{F} \cap Y) \leq \VC(\mathcal{F})$ and $\VC^*(\mathcal{F} \cap Y) \leq \VC^*(\mathcal{F})$.
\end{remark}

The following results is in stark contrast with the situation for the family of convex sets over the reals, where already the family of convex subsets of $\mathbb{R}^2$ has infinite VC-dimension (e.g.,~any set of points on a circle is shattered by the family of convex hulls of its subsets).

\begin{theorem}\label{thm: VC dim}
	Let $K$ be a valued field and $d \geq 1$. Then the family $\Conv_{K^{d}}$ has VC-dimension $d+1$.
\end{theorem}
\begin{proof}
We have $\VC \left(\Conv_{K^{d}} \right) \geq d+1$ since the set	 $E:=\left\{ 0,e_{1},\ldots,e_{d}\right\} \subseteq K^d$, with $e_i$ the $i$th vector of the standard basis, is shattered by $\Conv_{K^{d}}$. Indeed,
the convex hull of any subset is contained in its affine span, and for any $S \subseteq E$, $\aff(S)$ does not contain any of the points in $E \setminus S$. 

On the other hand, $\VC \left(\Conv_{K^{d}} \right) \leq d+1$ as no 
subset $Y$ of $K^d$ with $|Y| \geq d+2$ can be shattered by $\Conv_{K^{d}}$. Indeed, by  Proposition \ref{prop: Radon}, at least one of the points of $Y$ belongs to every convex set containing all the other points of $Y$.
\end{proof}

The dual VC-dimension of a family of sets is bounded by its breadth.
\begin{fact}\cite[Lemma 2.9]{aschenbrenner2016vapnik} \label{fac: breadth dual VC}
	Let $\mathcal{F} \subseteq \mathcal{P}(X)$ be a family of subsets of $X$ of breadth at most $d$. Then also $\VC^{*}(\mathcal{F}) \leq d$.
\end{fact}

Using it, we get the following:
\begin{theorem}\label{thm: dual VC-dim}
 For any valued field $K$ and $d \geq 1$, the family $\Conv_{K^{d}}$ has dual VC-dimension $d$.
\end{theorem}
\begin{proof}
The dual VC-dimension of $\Conv_{K^{d}}$ is at least
$d$ because the $d$ coordinate-aligned (convex) hyperplanes in $K^d$
generate a Boolean algebra with $2^{d}$ atoms. 

Conversely, the breadth of $\Conv_{K^{d}}$ is $d$ by Theorem \ref{thm: breadth d}, hence by Fact \ref{fac: breadth dual VC} its dual VC-dimension is also at most $d$.
\end{proof}

\begin{defn}\label{def: frac Helly numb}
\begin{enumerate}
	\item A family of sets $\mathcal{F} \subseteq \mathcal{P}(X)$ has \emph{fractional Helly number} $k \in \mathbb{N}_{\geq 1}$ if for every $\alpha \in \mathbb{R}_{>0}$ there exists $\beta \in \mathbb{R}_{>0}$ so that: for any $n \in \mathbb{N}$ and any sets $S_{1}, \ldots, S_{n}\in \mathcal{F}$ (possibly with repetitions), 
if there are $\geq\alpha{n \choose k}$ $k$-element subsets of the multiset $\left\{ S_{1}, \ldots, S_{n}\right\} $
with a non-empty intersection, then there are $\geq\beta n$ sets from
$\left\{ S_{1}, \ldots, S_{n}\right\} $ with a non-empty intersection. 
\item The \emph{fractional Helly number} of $\mathcal{F}$ refers to the minimal
$k$ with this property. Say that $\mathcal{F}$ has the \emph{fractional Helly
property} if it has a fractional Helly number.
\end{enumerate}
\end{defn} 

Note that any finite family of sets trivially has fractional Helly number $1$ by choosing $\beta$ sufficiently small with respect to the size of $\mathcal{F}$. We will use the following theorem of Matou\v{s}ek.
\begin{fact}\cite[Theorem 2]{matousek2004bounded}\label{fac: Matousek frac Helly}
Let $\mathcal{F} \subseteq \mathcal{P}(X)$ be a set system whose dual shatter function satisfies $\pi^*_{\mathcal{F}}(n) = o(n^k)$, i.e.~$\lim_{n \to \infty} \pi^*_{\mathcal{F}}(n) / n^k = 0$, where $k$ is a fixed integer. Then $\mathcal{F}$ has fractional Helly number $k$.
\end{fact}

\begin{remark}\label{rem: frac helly dependence}
Moreover, if $\VC^*(\mathcal{F}) = d < \infty$, then the fractional Helly number is $\leq d+1$, and the $\beta$ witnessing this can be chosen depending only on $d$ and $\alpha$ (and not on the family $\mathcal{F}$).

	Indeed, by Definition \ref{def: vc-dim and co}, if $\VC^*(\mathcal{F}) \leq d$, then $\pi^*_{\mathcal{F}}(n) \leq \left(\frac{e}{d}\right)^d n^d$ for all $n \geq d$, hence $\pi^*_{\mathcal{F}}(n) \leq c n^d$ for all $n \in \mathbb{N}$, where $c = c(d) := \left(\frac{e}{d}\right)^d + 2^d$.
	In particular we can choose $m = m(d,\alpha)$ so that $\pi^*_{\mathcal{F}}(m) < \frac{1}{4} \alpha {m \choose d+1}$. Then it follows from the proof of \cite[Theorem 2]{matousek2004bounded} that  $\beta = \frac{1}{2m}$ works for all $n \geq \frac{m}{\beta} = 2m^2$, and trivially $\beta = \frac{1}{2m^2}$ works for all $n \leq 2m^2$, hence $\beta := \beta(\alpha,d) :=\frac{1}{2m^2}$ works for all $n \in \mathbb{N}$.	
\end{remark}

Using this, we get the following:
\begin{theorem}\label{thm: fractional Helly number}
	If $K$ is a valued field, $d \geq 1$, and $X \subseteq K^d$ is an arbitrary subset, then the fractional Helly number of the family 
	\[\Conv_{K^{d}} \cap X = \left\{ C \cap X : C \in \Conv_{K^d}\right\} \subseteq \mathcal{P}(X)\]
	 is at most $d+1$. Moreover, $\beta$ in Definition \ref{def: frac Helly numb} can be chosen depending only on $d$ and  $\alpha$ (and not on the field $K$ or set $X$).
	And if $K$ is infinite, then the fractional Helly number of the family $\Conv_{K^d}$  is exactly $d+1$.
\end{theorem}
\begin{proof}
By Fact \ref{fac: Matousek frac Helly} we have that the fractional
Helly number of a set system is at most the smallest integer larger
than its dual VC-density. Dual VC-density is, in turn, at most its
dual VC-dimension. Also for any set $X \subseteq K^d$ we have $ \VC^{*} \left( \Conv_{K^d} \cap X \right) \leq \VC^{*} \left( \Conv_{K^d} \right)$ by Remark \ref{rem: VC dim of subfamily}.
So $\Conv_{K^{d}} \cap X$ has dual VC-density at
most $d$ by Theorem \ref{thm: dual VC-dim}, hence its fractional Helly number is at most $d+1$ by Fact \ref{fac: Matousek frac Helly}. And an appropriate $\beta$ can be chosen depending only on $d$ and $\alpha$ by Remark \ref{rem: frac helly dependence}.

To show  that the fractional Helly number of $\Conv_{K^{d}}$ is at least $d+1$ when $K$ is infinite, we can use
 the standard example with affine hyperplanes in general position. We include the details for completeness.
First note that as the field $K$ is infinite, for any $K$-vector space $V$ of dimension $k$ and $v \in V\setminus\{0\}$ there exists an infinite set $S \subseteq V$ so that $v \in S$ and any $k$ vectors from $S$ are linearly independent. This is clear for $k=1$ by taking any infinite set of non-zero vectors, so assume that $k > 1$. By induction on $i \in \mathbb{N}_{\geq k}$ we can find sets $S_i$ such that $v \in S_i, |S_i| \geq i$ and every $k$ vectors from $S_i$ are linearly independent, for all $i$. Let $S_k$ be any basis of $V$ containing $v$. Assume $i>k$ and $S_i$ satisfies the assumption. Since $K$ is infinite, $V$ is not a union of finitely many proper subspaces, in particular there exists some 
$$w \in V \setminus \bigcup_{s \subseteq S_i, |s| = k-1} \langle s \rangle.$$
Let $S_{i+1} := S_i \cup \{w\}$. Since in particular any $s \subseteq S_i$ with $|s| = k-1$ is linearly independent by the inductive assumption, it follows that $s \cup \{w\}$ is also linearly independent, hence $S_{i+1}$ satisfies the assumption. Finally, $S := \bigcup_{i \in \mathbb{N}_{\geq k}} S_i$ is as wanted.

In particular, we can find an infinite set of vectors $S$ in $K^{d}\times K$  so that any $d+1$ of them are linearly independent and the standard basis vector $e_{d+1} \in S$. Then 
$$X := \left\{ \langle v, - \rangle : v \in S \right\} \subseteq \left( K^{d}\times K \right)^*$$ is an infinite set of dual vectors such that any $d+1$ of them are linearly independent, and it contains the projection map onto the last coordinate $\pi_{d+1} := \langle e_{d+1}, - \rangle : \left(x_{1}, \ldots, x_{d+1}\right)\mapsto x_{d+1}$.
Consider the
family 
$$\mathcal{H}:=\left\{ \ker\left(f\right)\mid f\in X\setminus\left\{ \pi_{d+1}\right\} \right\} \subseteq \mathcal{P}\left( K^d \times K \right)$$
of kernels of these dual vectors (excluding the projection map onto
the last coordinate), and let
$$\mathcal{H}':=\left\{ \left\{ v\in K^{d}\mid\left(v,1\right)\in H\right\} \mid H\in \mathcal{H} \right\} \subseteq \mathcal{P} \left( K^d \right).$$
Then $\mathcal{H}'$ is an infinite family of affine hyperplanes in $K^{d}$,
and we wish to show that any $d$ element of $\mathcal{H}'$ intersect
in a point, and any $d+1$ elements of $\mathcal{H}'$ have empty
intersection. 
For any pairwise distinct $f_{1}, \ldots,f_{d} \in X\setminus\left\{ \pi_{d+1}\right\} $,
by linear independence
$$\dim\left(\ker\left(f_{1}\right)\cap \ldots \cap\ker\left(f_{d}\right)\right)=d+1-\dim\left(\langle f_{1}, \ldots,f_{d} \rangle\right)=1.$$
And by their linear independence with $\pi_{d+1}$, 
$$\dim\left(\ker\left(f_{1}\right)\cap \ldots \cap\ker\left(f_{d}\right)\cap\ker\left(\pi_{d+1}\right)\right)=0.$$
That is, $\ker\left(f_{1}\right)\cap \ldots \cap\ker\left(f_{d}\right)$
is a line in $K^d \times K$ that intersects $\ker\left(\pi_{d+1}\right)=K^{d}\times\left\{ 0\right\} $
only at the origin, and thus must also intersect $K^{d}\times\left\{ 1\right\} $
in a single point; this shows that every $d$ elements of $\mathcal{H}'$
intersect in a point. 
And any pairwise distinct $f_{1}, \ldots ,f_{d+1}\in X\setminus\left\{ \pi_{d+1}\right\} $
span $\left(K^{d}\times K\right)^{*}$ by linear independence, so
$\ker\left(f_{1}\right)\cap \ldots \cap\ker\left(f_{d+1}\right)=\left\{ 0\right\} $,
and thus has empty intersection with $K^{d}\times\left\{ 1\right\} $.
This shows that every $d+1$ elements of $\mathcal{H}'$ have empty
intersection.

Using $\alpha=1$, for any $\beta>0$, take an arbitrary $n\geq\frac{d+1}{\beta}$.
Let $H_{1}, \ldots, H_{n}\in \mathcal{H}'$ be any distinct hyperplanes from this collection.
All $d$-subsets (so, $\alpha{n \choose d}$ of them) of $\left\{ H_{1}, \ldots, H_{n}\right\} $
have an intersection point, but there are no $\beta n\geq d+1$ of
them with a common intersection point. Therefore $\text{Conv}_{K^{d}}$
does not have fractional Helly number $d$.
\end{proof}

Note that Theorems \ref{thm: Helly number} and \ref{thm: fractional Helly number} replicate results for real convex sets,
while Theorems \ref{thm: breadth d}, \ref{thm: VC dim}, and \ref{thm: dual VC-dim} do not: as we have already remarked, $\Conv_{\mathbb{R}^{2}}$
has infinite breadth, VC-dimension, and dual VC-dimension. 
The
following result is somewhere in between. The classical Tverberg theorem says that for any $X\subseteq\mathbb{R}^d$ with $|X|\geq(d+1)(r-1)+1$, $X$ can be partitioned into $r$ disjoint subsets $X_1,\ldots,X_r$ whose convex hulls intersect; that is, $\conv(X_1)\cap\ldots\cap\conv(X_r)\neq\emptyset$. Over valued fields, we obtain a much stronger version (note that any element of the non-empty set $X_r$ in the statement of theorem \ref{thm: Tverberg} belongs to the convex hulls of each of the sets $X_i, i \in [r]$ --- which gives the usual conclusion of Tverberg's theorem over the reals):

\begin{theorem}\label{thm: Tverberg}
	Let $K$ be a valued field and $d,r \in \mathbb{N}_{\geq 1}$. Then any set $X \subseteq K^d$ with
	 $$|X| \geq\left(d+1\right)\left(r-1\right)+1$$
points in $K^{d}$ can be partitioned into subsets $X_{1}, \ldots, X_{r}$
such that $\left|X_{i}\right|=d+1$ for $i<r$, $\left|X_{r}\right|=\left|X\right|-\left(d+1\right)\left(r-1\right)$,
and $\conv\left(X_{i}\right)\supseteq\conv\left(X_{j}\right)$
for all $i\leq j \in [r]$.
\end{theorem}
\begin{proof}
	Since any finitely generated convex set is the convex hull
of some $d+1$ points from it by Corollary \ref{cor: Carath}, we can find $X_{1} \subseteq X$ with $\left|X_{1}\right|=d+1$
and $\conv\left(X_{1}\right)=\conv\left(X\right)$, $X_{2} \subseteq X \setminus X_1$
with $\left|X_{2}\right|=d+1$ and $\conv\left(X_{2}\right)=\conv\left(X\setminus X_{1}\right)$,
and so on: once $X_{1}, \ldots, X_{i-1}$ have been chosen, pick $X_{i} \subseteq X \setminus \left( \bigcup_{j=1}^{i-1} X_j \right)$
such that $|X_i| = d+1$, $\conv\left(X_{i}\right)=\conv\left(X\setminus \bigcup_{j=1}^{i-1}X_j\right)$,
and then let $X_{r}$ consist of everything left over at the end.
\end{proof}

From this strong Tverberg theorem and the fractional Helly property,
we finally get an analog of the result due to Boros-F\"uredi \cite{boros1984number} and B\'ar\'any \cite{barany1982generalization} on the common points in the intersections of many ``simplices'' over valued fields (note that the conclusion is actually stronger than over the reals: the common point comes from the set $X$ itself). This answers a question asked by Kobi Peterzil and Itay Kaplan. Our argument is an adaptation of the second proof in \cite[Theorem 9.1.1]{matousek2013lectures}.

\begin{theorem}\label{thm: first selection lemma}
	For each $d \geq 1$ there is a constant $c = c(d) > 0$  such that: for any valued field $K$ and any finite $X\subseteq K^{d}$ (say $n:=\left|X\right|$), there
is some $a\in X$ contained in the convex hulls of at least $c{n \choose d+1}$
of the ${n \choose d+1}$ subsets of $X$ of size $d+1$.
\end{theorem}

\begin{proof}

Let $X \subseteq K^d$ with $|X| = n$ be given, and let 
$$\mathcal{F} := \Conv_{K^d} \cap X = \left\{ C \cap X : C \in \Conv_{K^d}  \right\}$$
 be the family of all subsets of $X$ cut out by the convex subsets of $K^d$. 
Let $\left( S_i \right)_{i \in [N]}$ with $S_i \in \Conv_{K^d}$ be the sequence listing all convex hulls of subsets of $X$ of size $d+1$ in an arbitrary order (possibly with repetitions). Then $N = {n \choose d+1}$, and for a $(d+1)$-element subset $Y \subseteq X$ we let $g(Y) \in [N]$ be the index at which $\conv(Y)$ appears in this sequence. For each $i \in [N]$ let $S'_i := S_i \cap X \in \mathcal{F}$.
It is thus sufficient to show that there exists some $\alpha >0$, depending only
on $d$, such that at least $\alpha{N \choose d+1}$ of the $\left(d+1\right)$-element subsets $I \subseteq [N]$ satisfy $\bigcap_{i \in I}S'_i  \neq \emptyset$ --- as then Theorem \ref{thm: fractional Helly number} applied to $\mathcal{F} \subseteq \mathcal{P}(X)$ shows
the existence of $c >0$ depending only on $\alpha,d$, and hence only on $d$, so that for some $I \subseteq [N]$ with $|I| \geq c N = c {n \choose d+1}$ there exists some $a \in \bigcap_{i \in I} S'_i \subseteq \bigcap_{i \in I} S_i$ (in particular $a \in X$).

Now we find an appropriate $\alpha$. For any $\left(d+1\right)^{2}$-element subset $Y\subseteq X$, by
Theorem \ref{thm: Tverberg} (with $r:=d+1$), we can fix a partition of $Y$ into $d+1$ disjoint parts
$Y_{1}, \ldots,Y_{d+1}$, each of which has $d+1$ elements, and so that $\conv(Y_i) \supseteq \conv(Y_j)$ for all $i \leq j \in [d+1]$. In particular any element of the non-empty set $Y_{[d+1]} \subseteq X$ belongs to $\bigcap_{i \in [d+1]} \left(\conv(Y_i) \cap X \right) = \bigcap_{i \in [d+1]} \left(S'_{g(Y_i)} \right)$. As $g$ is a bijection, $Y \mapsto \left\{ g(Y_i) : i \in [d+1] \right\}$
 gives a function $f$ from $\left(d+1\right)^2$-element subsets
of $X$ to $\left(d+1\right)$-element subsets $I \subseteq [N]$ so that $\bigcap_{i \in I} S'_i \neq \emptyset$. Moreover, $f$ is an injection. Indeed, given a set $\{ j_i : i \in [d+1] \}$ in the image of $f$, as $g$ is a bijection, there is a unique set $\left\{Y_1, \ldots, Y_{d+1} \right\}$ with $Y_i \subseteq X$ disjoint of size $d+1$ so that $g(Y_i) = j_i$ for all $i \in [d+1]$, and there can be only one set $Y \subseteq X$ of size $(d+1)^2$ for which it is a partition. If follows that the number
of sets $I \subseteq [N]$ with $\bigcap_{i \in I}S'_i \neq \emptyset$ is at least 
$$\binom{n}{(d+1)^2} = \Omega \left( n^{(d+1)^2} \right) \geq \alpha \binom{N}{d+1} $$
for some sufficiently small $\alpha$ depending only on $d$.
%The  probability that a random $\left(d+1\right)$-subset
%of $[N]$ is in the image of this map is the probability that
%the $d+1$ sets are disjoint divided by the number of partitions of
%a $\left(d+1\right)^{2}$-element set into $d+1$ subsets of size
%$d+1$ each. The former is positive and converges to $1$ as $n\rightarrow\infty$,
%so it has a positive lower bound, and the latter does not depend on
%$n$. So this gives us a positive lower bound $\alpha$ for the probability
%that a random $\left(d+1\right)$-subset of $[N]$ is in the
%image of this map, so $\alpha$ is also a lower bound for the probability
%that a random $\left(d+1\right)$-subset of $[N]$ has nonempty
%intersection.
\end{proof}

\section{Final remarks and questions}\label{sec: final remarks}
\subsection{Some further results and future directions}\label{sec: final quest and dirs}
The results of Section \ref{sec: combinatorics of conv sets} imply the following analog of the celebrated \emph{$(p,q)$-theorem} of Alon and Kleitman \cite{alon1992piercing} for convex sets over valued fields.
\begin{cor}\label{cor: pq theorem}
	For any $d,p,q \in \mathbb{N}_{\geq 1}$ with $p \geq q \geq d+1$ there exists $T = T(p,q,d) \in \mathbb{N}$ such that: if $K$ is a valued field and $\mathcal{F}$ is a family of convex subsets of $K^d$ such that among every $p$ sets of $\mathcal{F}$, some $q$ have a non-empty intersection, then there exists a $T$-element set $Y \subseteq K^d$ intersecting all sets of $\mathcal{F}$. 
	\end{cor}
\noindent Corollary \ref{cor: pq theorem} follows formally by applying \cite[Theorem 8]{alon2002transversal} since the family $\Conv_{K^d}$ has fractional Helly property (Theorem \ref{thm: fractional Helly number})  and is closed under intersections. Alternatively, it follows with a slightly better bound on $T$ by combining the fractional Helly property with the existence of $\varepsilon$-nets for families of bounded VC-dimension (Theorem \ref{thm: VC dim}), as outlined at the end of \cite[Section 1]{matousek2004bounded}. The problem of determining the optimal bound on $T(p,q,d)$ is widely open over the reals (see \cite[Section 2.6]{barany2021helly}), and we expect that it might be easier  in the valued fields setting.

Kalai \cite{kalai1984intersection} and Eckhoff \cite{eckhoff1985upper} proved that in the fractional Helly property for convex sets over the reals, one can take $\beta(d,\alpha) = 1 - (1-\alpha)^{\frac{1}{d+1}}$ (and this bound is sharp). 

\begin{problem}
	What is the optimal dependence of $\beta$ on $d,\alpha$ in Theorem \ref{thm: fractional Helly number}? 
\end{problem}
Over $\mathbb{R}$, Sierksma's Dutch cheese conjecture predicts a lower bound for the number of Tverberg partitions (see e.g.~\cite[Conjecture 3.12]{de2019discrete} and the references there). We expect the same bound to holds over valued fields:

\begin{conj}\label{conj: Dutch cheese}
For any valued field $K$ and $X\subset K^d$ with $|X|=(r-1)(d+1)+1$, there are at least $((r-1)!)^d$ partitions of $X$ into parts whose convex hulls intersect.
\end{conj}

\begin{remark}
In Theorem \ref{thm: Tverberg}, we showed the existence of Tverberg partitions satisfying the stronger property that the convex hulls of the parts are linearly ordered by inclusion. It is not true that for $X\subseteq K^d$ with $|X|=(d+1)(r-1)+1$, there are at least $((r-1)!)^d$ different ways of partitioning $X$ into $X_1,\ldots,X_r$ such that $\conv(X_1)\supseteq\ldots\supseteq\conv(X_r)$. Thus any attempt to prove Conjecture \ref{conj: Dutch cheese} would have to involve other Tverberg partitions that do not have this property.
For an example in $K^2$ where this bound fails, let $x\in K$ with $\nu(X)\neq0$, and let $X:=\{(x^n, x^{-n}) | n \in [3(r-1)+1]\}$. For any partition of $X$ into $X_1,\ldots,X_r$ such that $\conv(X_1)\supseteq\ldots\supseteq\conv(X_r)$, for each $i<r$, $X_i$ must consist of the points corresponding to the lowest and highest values of $n$ among all points not already in $X_1\cup\ldots\cup X_{i-1}$, together with one of the other $3(r-i)-1$ remaining points, and $X_r$ must consist of whatever point is left over. So the number of partitions of $X$ of this form is $\prod_{i=1}^{r-1}(3(r-i)-1)<\prod_{i=1}^{r-1}3(r-i)=3^{r-1}(r-1)!<((r-1)!)^2$ for large enough $r$.
\end{remark}

We expect that the \emph{colorful} Tverberg theorem also holds
over valued fields, however the proofs for convex sets over $\mathbb{R}$ rely on topological arguments not readily available in the valued field context:

\begin{conj}
	For any integers $r,d \geq 2$ there exists 
$t\geq r$ such that: for any valued field $K$ and $X\subseteq K^{d}$ with $\left|X\right|=t\left(d+1\right)$,
partitioned into $d+1$ color classes $C_{1}, \ldots, C_{d+1}$ each of size
$t$, there exist pairwise disjoint $X_{1}, \ldots, X_{r}\subseteq X$ with $\left|X_{i}\cap C_{j}\right|=1$
for $i\in\left[r\right]$ and $j\in\left[d+1\right]$, and $\bigcap_{i\in\left[r\right]}\conv\left(X_{i}\right)\neq\emptyset$.
\end{conj}
\noindent It would formally imply (see e.g.~\cite[Section 9.2]{matousek2013lectures}) the ``second selection lemma'' over valued fields generalizing Theorem \ref{thm: first selection lemma}:
\begin{conj}
	For each $d \in \mathbb{N}_{\geq 1}$ there exist $c,s>0$ such that: for any valued field $K$, 
$\alpha\in\left(0,1\right]$ and $n \in \mathbb{N}$, for every $X\subseteq K^{d}$
with $\left|X\right|=n$, and every family $\mathcal{F}$ of $(d+1)$-element subsets
of $X$ with $\left| \mathcal{F} \right|\geq\alpha{n \choose d+1}$, there
is a point contained in the convex hulls of at least $c\alpha^{s}{n \choose d+1}$
of the elements of $\mathcal{F}$.
\end{conj}

% Strong EH holds for bipartite incidence graphs of convex sets - since definable in Sh, distal expansion of $Q_p$ (without bipartite, holds for convex sets in the real plane, by http://math.mit.edu/~fox/paper-TuranAndErdosHajnal082607.pdf) Reduct to the uniform family of convex sets - always has distal expansions? Or at least locally distal?

	Corollary \ref{cor: descr of conv sets in sph compl} has the following immediate model-theoretic application.
	\begin{remark}\label{rem: conv ext def}
		If $K$ is a spherically complete valued field, then every convex
subset of $K^{d}$ is definable in the expansion of the field $K$ by a predicate
for each Dedekind cut of the value group (so in particular definable in \emph{Shelah expansion of $K$} by all externally definable sets \cite{shelah2009dependent, chernikov2013externally}). And conversely, every Dedekind
cut of the value group is definable in the expansion of $K$ by
a predicate for each $\mathcal{O}$-submodule of $K$. In particular,
if $K$ has value group $\mathbb{Z}$, then all convex subsets of
$K^{d}$ form a definable family.
	\end{remark}
	
\begin{expl}\label{ex: naming convex over reals}
In contrast, naming a single (bounded) convex subset of $\mathbb{R}^2$ in the field of reals allows to define the set of integers.
 Indeed, we can define a continuous and piecewise linear function $f:[0,1] \to [0,1]$  such that 
$$C := \left\{(x,y): x\in [0,1], 0\leq y\leq f(x) \right\}$$
 is convex but the set of 
points where $f$ is not differentiable is exactly $\left\{\frac{1}{n} : n \in \mathbb{N}_{\geq 2} \right\}$. 
Now in the field of reals with a predicate for $C$ we can define $f$ and the set of points where it is not differentiable, hence $\mathbb{N}$ is also definable.	
\end{expl}

\subsection{Other notions of convexity over non-archimedean fields} \label{sec: other convexities} We briefly overview several  other kinds of convexities over non-archimedean fields considered in the literature. The extension of
Hilbert (projective) geometry to convex sets in a generalized sense is a topic of high current interest,
see e.g.~\cite{guilloux2016yet}. In a different spirit, in tropical geometry, convex sets
over real closed non-archimedean fields have been considered (unlike what is done here, this leads
to a combinatorial convexity similar to the classical one, since by Tarski's completeness theorem,
polyhedral properties of a combinatorial nature are the same over all real closed fields). Moreover,
tropical polyhedra are obtained as images of such polyhedra by the nonarchimedean valuation, see e.g.~\cite{develin2007tropical}.
Polytopes and simplexes in p-adic fields are introduced in \cite{darniere2017polytopes, darniere2019semi}, and demonstrated to play in $p$-adically closed fields the role played by real simplexes in the classical results of triangulation of semi-algebraic sets over real closed fields.
Although we are not aware of any direct link of these results with the present
work, we hope for some connections to be found in the future.

\subsection{Abstract convexity spaces}\label{sec: convexity spaces}
Our results here can be naturally placed in the context of abstract convexity spaces, we refer to e.g.~\cite{van1993theory} for an introduction to the subject. A \emph{convexity space} is a pair $(X, \mathcal{C})$, where $X$ is a set and $\mathcal{C} \subseteq 2^{X}$ is a family of subsets of $X$ closed under intersection with $\emptyset, X \in \mathcal{C}$. The sets in $\mathcal{C}$ are called \emph{convex}. Given a subset $Y \subseteq X$, the \emph{convex hull} of $Y$, denoted $\conv(Y)$, is the smallest set in $\mathcal{C}$ containing $Y$ (equivalently, the intersection of all sets in $\mathcal{C}$ containing $Y$). A convex set $C \in \mathcal{C}$ is called a \emph{half-space} if its complement is also convex. The convexity space $(X, \mathcal{C})$ is \emph{separable} if for every $C \in \mathcal{C}$ and $x \in X \setminus C$, there exists a half-space $H \in \mathcal{C}$ so that $C \subseteq H$ and $x \notin H$ (equivalently, if  every convex set is the intersection of all half-spaces containing it). Separability is an abstraction of the hyperplane separation (and more generally Hahn-Banach) theorem. In particular, $\left(\mathbb{R}^d, \Conv_{\mathbb{R}^d} \right)$ is a separable convexity space (see e.g.~\cite[Section 1.1]{moran2019weak} or \cite{van1993theory} for many other examples). The \emph{Radon number}\footnote{Sometimes in the literature it is defined with ``$\geq$'' instead of ``$>$'' leading to the value off by $1$, we are following the notation from \cite[Chapter II]{van1993theory} here.}  of a convexity space $(X, \mathcal{C})$  is the smallest $k \in \mathbb{N}_{\geq 1}$ (if it exists) such that every $Y \subseteq X$ with $|Y| > k$ can be partitioned into two parts $Y_1, Y_2$ such that $\conv(Y_1) \cap \conv(Y_2) \neq \emptyset$ (the classical Radon's theorem states that the Radon number of $\left(\mathbb{R}^d, \Conv_{\mathbb{R}^d} \right)$ equals $d+1$). Given $\emptyset \neq Y \subseteq X$, a partition $Y_1, \ldots, Y_r$ of $Y$ is \emph{Tverberg} if $\bigcap_{i=1}^{r} \conv(Y_i) \neq \emptyset$. The \emph{$r$th Tverberg number} of $(X, \mathcal{C})$ is the smallest $k$ so that every $Y \subseteq X$ with $|Y| > k$ has a Tverberg partition in $r+1$ parts. Note that the first Tverberg number is the Radon number, and the classical theorem of Tverberg says that the $r$th Tverberg number of $\left( \mathbb{R}^d, \Conv_{\mathbb{R}^d} \right) $ is $r(d+1)$.

Now let $K$ be a valued field and $d \in \mathbb{N}_{\geq 1}$. Then $\left(K^d, \Conv_{K^d} \right)$ is a convexity space, but we stress that it is \emph{not separable}; in fact, $\emptyset$ and $K^d$ are the only half-spaces. This is because for any non-empty proper convex set $C$, let $x\in C$, $y\in K^d\setminus C$, and $\alpha\in K\setminus\mathcal{O}$. Then $z:=x+\alpha(y-x)\notin C$, since $y=\alpha^{-1}z+(1-\alpha^{-1})x$ is a convex combination. But then $x=(1-\alpha)^{-1}(z-\alpha y)$ is a convex combination of elements of $K^d\setminus C$, so $K^d\setminus C$ is not convex.

Proposition \ref{prop: Radon} implies that the Radon number of $\left(K^d, \Conv_{K^d} \right)$ is $d+1$. By the Levi inequality in an arbitrary convexity space (\cite[Chapter II(1.9)]{van1993theory}), it follows that the Helly number of $\Conv_{K^d}$ (Definition \ref{def: Helly num}) is $\leq d+1$ (we included a proof in Theorem \ref{thm: Helly number} for completeness).
It was also recently shown in \cite{holmsen2021radon} that in any convexity space $(X, \mathcal{C})$ with Radon number $k$, $\mathcal{C}$ has a fractional Helly number (Definition \ref{def: frac Helly numb}) bounded by some function of $k$. In the case of $\left(K^d, \Conv_{K^d} \right)$ this general bound is much weaker than the optimal bound $d+1$ given in Theorem \ref{thm: fractional Helly number}. Corollary \ref{cor: Carath} implies that the Carath\'eodory number of $\left(K^d, \Conv_{K^d} \right)$ is $d+1$ (see \cite[Chapter II(1.5)]{van1993theory} for the definition). Finally, Theorem \ref{thm: Tverberg} implies that the $r$th Tverberg number of $\left(K^d, \Conv_{K^d} \right)$  is $r(d+1)$ (finiteness of the $r$th Tverberg numbers for all $r$ follows from the finiteness of the Radon number in an arbitrary convexity space, with a much weaker bound \cite[Chapter II(5.2)]{van1993theory}).

\bibliographystyle{alpha}
\bibliography{ref}

\newcommand{\etalchar}[1]{$^{#1}$}
\begin{thebibliography}{AvdDvdH17}

\bibitem[ADH{\etalchar{+}}16]{aschenbrenner2016vapnik}
Matthias Aschenbrenner, Alf Dolich, Deirdre Haskell, Dugald Macpherson, and
  Sergei Starchenko.
\newblock Vapnik-{C}hervonenkis density in some theories without the
  independence property, {I}.
\newblock {\em Transactions of the American Mathematical Society},
  368(8):5889--5949, 2016.

\bibitem[AK92]{alon1992piercing}
Noga Alon and Daniel~J Kleitman.
\newblock Piercing convex sets and the {H}adwiger-{D}ebrunner $(p, q)$-problem.
\newblock {\em Advances in Mathematics}, 96(1):103--112, 1992.

\bibitem[AKMM02]{alon2002transversal}
Noga Alon, Gil Kalai, Ji\v{r}\'i Matou\v{s}ek, and Roy Meshulam.
\newblock Transversal numbers for hypergraphs arising in geometry.
\newblock {\em Advances in Applied Mathematics}, 29(1):79--101, 2002.

\bibitem[AvdDvdH17]{aschenbrenner2017asymptotic}
Matthias Aschenbrenner, Lou van~den Dries, and Joris van~der Hoeven.
\newblock {\em Asymptotic differential algebra and model theory of
  transseries}.
\newblock Princeton University Press, 2017.

\bibitem[B{\'a}r82]{barany1982generalization}
Imre B{\'a}r{\'a}ny.
\newblock A generalization of {C}arath{\'e}odory's theorem.
\newblock {\em Discrete Mathematics}, 40(2-3):141--152, 1982.

\bibitem[BF84]{boros1984number}
Endre Boros and Zolt{\'a}n F{\"u}redi.
\newblock The number of triangles covering the center of an n-set.
\newblock {\em Geometriae Dedicata}, 17(1):69--77, 1984.

\bibitem[BK22]{barany2021helly}
Imre B{\'a}r{\'a}ny and Gil Kalai.
\newblock Helly-type problems.
\newblock {\em Bulletin of the American Mathematical Society}, 59(4):471--502,
  2022.

\bibitem[CS13]{chernikov2013externally}
Artem Chernikov and Pierre Simon.
\newblock Externally definable sets and dependent pairs.
\newblock {\em Israel Journal of Mathematics}, 194(1):409--425, 2013.

\bibitem[Dar17]{darniere2017polytopes}
Luck Darniere.
\newblock Polytopes and simplexes in p-adic fields.
\newblock {\em Annals of Pure and Applied Logic}, 168(6):1284--1307, 2017.

\bibitem[Dar19]{darniere2019semi}
Luck Darni{\`e}re.
\newblock Semi-algebraic triangulation over p-adically closed fields.
\newblock {\em Proceedings of the London Mathematical Society},
  118(6):1501--1546, 2019.

\bibitem[DLGMM19]{de2019discrete}
Jes{\'u}s De~Loera, Xavier Goaoc, Fr{\'e}d{\'e}ric Meunier, and Nabil Mustafa.
\newblock The discrete yet ubiquitous theorems of {C}arath{\'e}odory, {H}elly,
  {S}perner, {T}ucker, and {T}verberg.
\newblock {\em Bulletin of the American Mathematical Society}, 56(3):415--511,
  2019.

\bibitem[DY07]{develin2007tropical}
Mike Develin and Josephine Yu.
\newblock Tropical polytopes and cellular resolutions.
\newblock {\em Experimental Mathematics}, 16(3):277--291, 2007.

\bibitem[Eck85]{eckhoff1985upper}
J{\"u}rgen Eckhoff.
\newblock An upper-bound theorem for families of convex sets.
\newblock {\em Geometriae Dedicata}, 19(2):217--227, 1985.

\bibitem[Fuc75]{fuchs1975vector}
L{\'a}szl{\'o} Fuchs.
\newblock Vector spaces with valuations.
\newblock {\em Journal of Algebra}, 35(1-3):23--38, 1975.

\bibitem[Gui16]{guilloux2016yet}
Antonin Guilloux.
\newblock Yet another $ p $-adic hyperbolic disc: Hilbert distance for $ p
  $-adic fields.
\newblock {\em Groups, Geometry, and Dynamics}, 10(1):9--43, 2016.

\bibitem[HL21]{holmsen2021radon}
Andreas~F Holmsen and Donggyu Lee.
\newblock Radon numbers and the fractional helly theorem.
\newblock {\em Israel Journal of Mathematics}, 241(1):433--447, 2021.

\bibitem[Hru14]{hrushovski2014imaginaries}
Ehud Hrushovski.
\newblock Imaginaries and definable types in algebraically closed valued
  fields.
\newblock {\em Valuation Theory in Interaction}, pages 297--319, 2014.

\bibitem[Joh16]{johnson2016fun}
William~Andrew Johnson.
\newblock {\em Fun with fields}.
\newblock PhD thesis, UC Berkeley, 2016.

\bibitem[Kal84]{kalai1984intersection}
Gil Kalai.
\newblock Intersection patterns of convex sets.
\newblock {\em Israel Journal of Mathematics}, 48(2-3):161--174, 1984.

\bibitem[Mat02]{matousek2013lectures}
Ji\v{r}\'{\i} Matou\v{s}ek.
\newblock {\em Lectures on discrete geometry}, volume 212 of {\em Graduate
  Texts in Mathematics}.
\newblock Springer-Verlag, New York, 2002.

\bibitem[Mat04]{matousek2004bounded}
Ji\v{r}\'i Matou\v{s}ek.
\newblock Bounded {VC}-dimension implies a fractional {H}elly theorem.
\newblock {\em Discrete \& Computational Geometry}, 31(2):251--255, 2004.

\bibitem[Mon46]{Monna}
A.~F. Monna.
\newblock Sur les espaces lin\'{e}aires norm\'{e}s. {I}, {II}, {III}, {IV}.
\newblock {\em Nederl. Akad. Wetensch., Proc.}, 49, 1946.

\bibitem[MY19]{moran2019weak}
Shay Moran and Amir Yehudayoff.
\newblock On weak $\epsilon$-nets and the {R}adon number.
\newblock In {\em 35th International Symposium on Computational Geometry (SoCG
  2019)}. Schloss Dagstuhl-Leibniz-Zentrum fuer Informatik, 2019.

\bibitem[PGS10]{perez2010locally}
C.~P{\'e}rez-Garci{\'a} and Wilhelmus~Hendricus Schikhof.
\newblock {\em Locally convex spaces over non-Archimedean valued fields}.
\newblock Cambridge Univ. Press, 2010.

\bibitem[Sch50]{schilling1950theory}
Otto Franz~Georg Schilling.
\newblock {\em The theory of valuations}, volume~4.
\newblock American Mathematical Soc., 1950.

\bibitem[Sch13]{schneider2013nonarchimedean}
Peter Schneider.
\newblock {\em Nonarchimedean functional analysis}.
\newblock Springer Science \& Business Media, 2013.

\bibitem[She09]{shelah2009dependent}
Saharon Shelah.
\newblock Dependent first order theories, continued.
\newblock {\em Israel Journal of Mathematics}, 173(1):1, 2009.

\bibitem[vdD14]{van2014lectures}
Lou van~den Dries.
\newblock Lectures on the model theory of valued fields.
\newblock In {\em Model theory in algebra, analysis and arithmetic}, pages
  55--157. Springer, 2014.

\bibitem[vDV93]{van1993theory}
M.L.J. van De~Vel.
\newblock {\em Theory of convex structures}.
\newblock Elsevier, 1993.

\end{thebibliography}
\end{document}